\documentclass{amsart}

\usepackage{amssymb} \usepackage{amsfonts} \usepackage{amsmath}
\usepackage{amsthm} \usepackage{epsfig} 
\usepackage{color}
\usepackage{pinlabel}
\usepackage{amscd}
\usepackage[all]{xy}
\usepackage{pinlabel}
\usepackage{tcolorbox}
\usepackage{seqsplit}

\newtheorem{lemma}{Lemma}
\newtheorem{teo}[lemma]{Theorem}
\newtheorem{prop}[lemma]{Proposition}
\newtheorem{cor}[lemma]{Corollary}

\theoremstyle{definition}

\newtheorem{quest}{Question}
\newtheorem{probl}[quest]{Problem}

\theoremstyle{remark}

\newcommand{\matr} [4] {\big({\tiny\begin{array}{@{}c@{\ }c@{}} #1 & #2 \\ #3 & #4 \\ \end{array}} \big)}

\newcommand{\interior}[1]{{\rm int}(#1)}

\newcommand{\Ivan}{Ivan\v si\' c\ }

\newcommand{\matR} {\ensuremath {\mathbb{R}}}

\newcommand{\matZ} {\ensuremath {\mathbb{Z}}}

\newcommand{\matH} {\ensuremath {\mathbb{H}}}

\newcommand{\matRP} {\ensuremath {\mathbb{RP}}}

\newcommand{\SO} {\ensuremath {{\rm SO}}}

\newcommand{\Vol} {\ensuremath {{\rm Vol}}}

\newcommand{\GL} {\ensuremath {{\rm GL}}}

\author{Bruno Martelli}
\address{Dipartimento di matematica, Largo Pontecorvo 5, 56126 Pisa, Italy}
\email{bruno dot martelli at unipi dot it}

\title{Five tori in $S^4$}

\begin{document}

\begin{abstract}
Ivan\v si\' c proved that there is a link $L$ of five tori in $S^4$ with hyperbolic complement. We describe $L$ explicitly with pictures, study its properties, and discover that $L$ is in many aspects similar to the Borromean rings in $S^3$. In particular the following hold:  
\begin{enumerate}
\item Any two tori in $L$ are unlinked, but three are not; 
\item The complement $M=S^4\setminus L$ is integral arithmetic hyperbolic;
\item The symmetry group of $L$ acts $k$-transitively on its components for all $k$;
\item The double branched covering over $L$ has geometry $\matH^2 \times \matH^2$; 
\item The fundamental group of $M$ has a nice presentation via commutators; 
\item The Alexander ideal has an explicit simple description; 
\item Every class $x\in H^1(M, \matZ)=\matZ^5$ with $x_i\neq 0$ is represented by a perfect circle-valued Morse function; 
\item By longitudinal Dehn surgery along $L$ we get a closed 4-manifold with fundamental group $\matZ^5$;
\item The link $L$ can be put in perfect position.
\end{enumerate}

This leads also to the first descriptions of a cusped hyperbolic 4-manifold as a complement of tori in $\matRP^4$ and as a complement of some explicit Lagrangian tori in the product of two surfaces of genus two.
\end{abstract}

\maketitle

\section*{Introduction}

The aim of this paper is to investigate a link of five tori in $S^4$, first discovered by Ivan\v si\' c \cite{I,I2}. We describe this link explicitly with pictures and study its properties. It turns out that this link is in many aspects similar to the Borromean rings in $S^3$. 
We work in the smooth category: all (sub-)manifolds and maps are smooth, except when explicitly stated otherwise.

\subsection*{The Borromean rings}
The \emph{Borromean rings} $B \subset S^3$ shown in Figure \ref{Borromean:fig} have many well-known remarkable features: 
\begin{enumerate}
\item Any two components of $B$ form a trivial link;
\item The complement $N = S^3 \setminus B$ has an integral arithmetic hyperbolic metric;
\item We may isotope $B$ so that it is preserved by a group of 48 isometries of $S^3$ that acts transitively on its components;
\item The branched double covering of $S^3$ ramified over $B$ is a flat 3-manifold;
\item The group $\pi_1(N)$ is generated by some meridians $a,b,c$ with relators
$$[a,[b,c^{-1}]], \quad [b,[c,a^{-1}]], \quad [c,[a,b^{-1}]].$$
The pair $a$, $[b,c^{-1}]$ generates a peripheral subgroup $\matZ^2$, and the same holds for the other pairs obtained by permuting cyclically the letters $a,b,c$;
\item The Alexander ideal $I$ is generated by monomials
$(t_1-1)^{a_1}(t_2-1)^{a_2}(t_3-1)^{a_3}$
with $\{a_1,a_2,a_3\} = \{1,1,2\}$. Therefore the Alexander polynomial is
$$\Delta = (t_1-1)(t_2-1)(t_3-1);$$
\item A class $\phi = (x_1,x_2,x_3)\in H^1(N,\matZ) = \matZ^3$ fibers $\Longleftrightarrow$ $x_i \neq 0$ for all $i$; 
\item The longitudinal Dehn surgery along $B$ produces the 3-torus;
\item The rings $B$ may be put in prefect position.
\end{enumerate}

\begin{figure}
\centering
\includegraphics[width=4 cm]{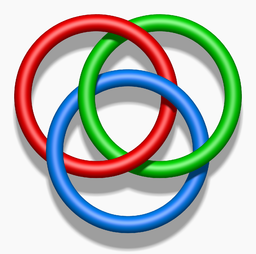}
\caption{The Borromean rings $B$ in $S^3$.}\label{Borromean:fig}
\end{figure}

All these facts are well-known. The integral arithmetic hyperbolic structure arises from the description of the complement as $N= \matH^3/\Gamma$ for some lattice $\Gamma < \SO^+(3,1) \cap \GL(4,\matZ)$, see for instance \cite{RT}. The 48 isometries are constructed in Proposition \ref{B:isom:prop}.
The flat double branched covering is the Hantsche--Wendt flat 3-manifold \cite{Z}. 

The \emph{Alexander ideal} $I\subset\matZ[H^1(N,\matZ)] = \matZ[t_1^{\pm 1}, t_2^{\pm 1}, t_3^{\pm 1}]$ was introduced by Fox \cite{F}, and the \emph{Alexander polynomial} $\Delta$ is its greater common divisor, well-defined up to multiplication with a product of monomials $\pm t_i^{\pm 1}$, see also \cite{McM}. The Thurston ball of $N$ is an octahedron with vertices $\pm e_i$ and all its faces are fibered \cite{Th}, hence $\phi = (x_1,x_2,x_3)$ fibers $\Longleftrightarrow$ it projects in the interior of a face $\Longleftrightarrow$ $x_i \neq 0$ for all $i$. 

We say that a link in $\matR^3$ is in \emph{perfect position} if the height function on each component has only two singular points, one minimum and one maximum; this implies in particular that every component of the link is unknotted. The diagram in Figure \ref{Borromean:fig} displays $B$ in perfect position. 


\subsection*{The Ivan\v si\' c link}
We now define the main protagonist of this paper, that is a particularly interesting link $L\subset S^4$ of five tori. We call this link the \emph{\Ivan link}, since its existence was proved by \Ivan in \cite{I, I2}, although it was not described explicitly there. We describe the \Ivan link $L\subset S^4$ as a \emph{film} in $S^3$. The film shows the intersections of $L$ with the slices $x_5 = t$ of $S^4 \subset \matR^5$ as $t \in (-1,1)$ varies. The five tori in $L$ are coloured in \emph{black, red, blue, green}, and \emph{gray} in the pictures.

The intersection between the link $L$ and the equator $S^3 = \{x_5=0\}$ is shown in Figure \ref{torus_all:fig}. It consists of one black torus, and a link of 8 circles. The black torus is entirely contained in the equator, while the red, blue, green, and gray tori intersect it transversely, each into a pair of circles with the same colour. Note that the black torus decomposes $S^3$ into two solid tori, and each solid torus intersects two components of $L$ into an \emph{untwisted chain link} with four components.

\begin{figure}
\centering
\includegraphics[width=9 cm]{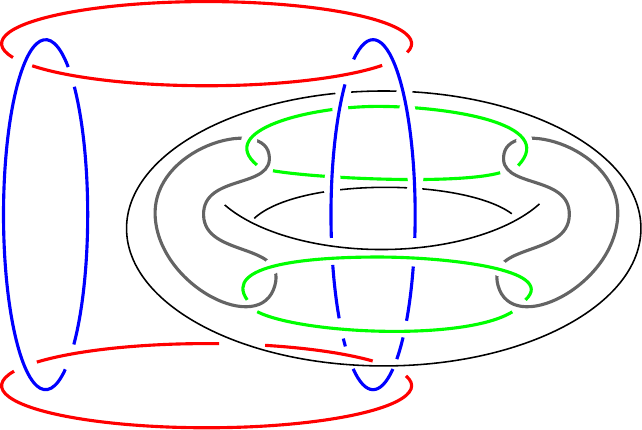}
\caption{The equatorial 3-sphere $S^3 = \{x_5=0\}$ intersects the link $L$ in one black torus and 8 circles.}\label{torus_all:fig}
\end{figure}

Each pair of red, blue, green, or gray circles in Figure \ref{torus_all:fig} forms an unlink, and hence bounds a properly embedded annulus in the equator $S^3 = \{x_5=0\}$. By pushing the interiors of two copies of this annulus in the interiors of the two hemispheres $x_5\geq 0$ and $x_5 \leq 0$ we get two annuli properly embedded in the hemispheres, whose union is a torus, that intersects the equator transversely into the two circles. A fact that is absolutely non-obvious from the figure, is that one can find four \emph{disjoint} properly embedded annuli (in each hemisphere) for the four pairs of red, blue, green and gray circles, and thus get four disjoint tori in $S^4$.

\begin{figure}
\centering
\labellist
\pinlabel $t=0$ at 330 450
\pinlabel $t=\pm 0.2$ at 200 220
\pinlabel $t=\pm 0.4$ at 425 220
\pinlabel $t=\pm 0.6$ at 200 5
\pinlabel $t=\pm 0.8$ at 425 5
\endlabellist
\includegraphics[width=12 cm]{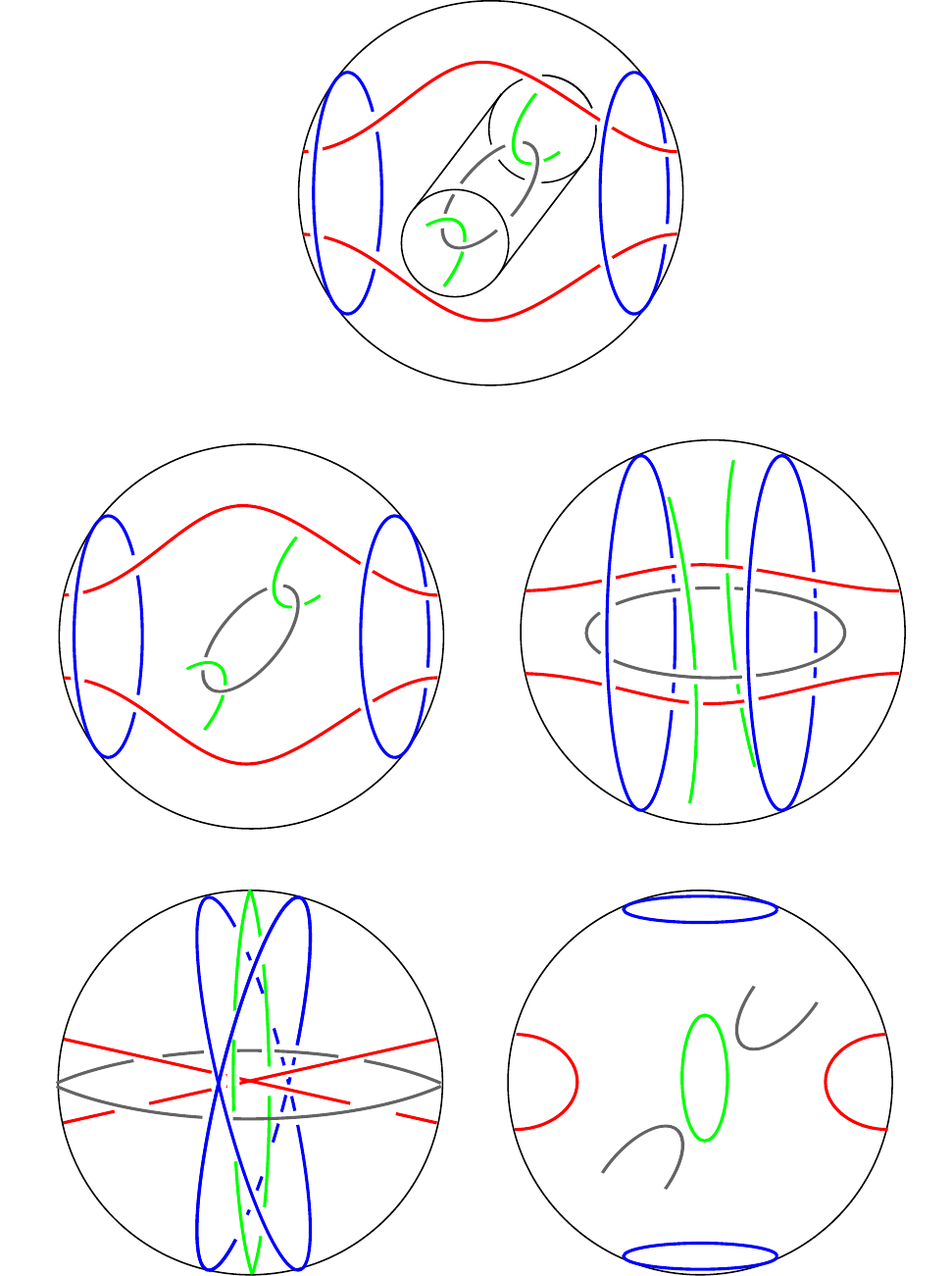}
\caption{The intersection of $L$ with the slice $x_5=t$ is obtained by doubling the corresponding picture in $D^3$ along the boundary $\partial D^3 = S^2$. At the times $t=\pm 0.9$ the picture consists of eight points, obtained by shrinking the eight circles shown at $t=\pm 0.8$.
}\label{tori:fig}
\end{figure}

These four disjoint annuli are described as a film in Figure \ref{tori:fig}. For simplicity, we see Figure \ref{torus_all:fig} as the double of the picture in Figure \ref{tori:fig}-(top). This picture shows a 3-disc $D^3$ that contains a properly embedded black annulus, two blue circles in $\partial D^3$, two properly embedded red and green arcs in $D^3$, and one gray circle in $\interior {D^3}$. The reader may check easily that by doubling this picture along $\partial D^3 = S^2$ we get the picture in Figure \ref{torus_all:fig}, where the black annulus doubles to the black torus, the properly embedded arcs become circles, the gray circle doubles to two gray circles, and the blue circles stay as they are.

Figure \ref{tori:fig} then shows a film in $D^3$ at various times $t\in (-1,1)$. The intersection of $L$ with the slices $x_5=t$ is then obtained by doubling the corresponding picture in $D^3$. The black torus is fully contained in the equator $x_5=0$. Each red, blue, green and gray torus intersects the slice $x_5=t$ transversely, except at the times $t=\pm 0.6$ and $t=\pm 0.9$. The function $x_5$ restricted to each such torus is a Morse function with 
with two minima at $t=-0.9$, two saddles at $t=-0.6$, two more saddles at $t=0.6$, and two maxima at $t=0.9$. 

We note that all the pictures in Figure \ref{tori:fig} are symmetric with respect to reflections along three coordinate planes; the link $L$ is also by construction symmetric with respect to reflections along $\partial D^3$ and $x_5=0$. Therefore $L\subset S^4$ is symmetric with respect to reflections along all the hyperplanes $x_i=0$ for $i=1,\ldots, 5$. All these symmetries preserve the components of $L$: as we will see, the link has also more symmetries that exchange its components.

\subsection*{Properties of the link}
As mentioned at the beginning of this introduction, the Ivan\v si\' c link $L\subset S^4$ shares many remarkable properties with the Borromean rings. The following is the main result of this paper. 

\begin{teo} \label{main:teo}
The Ivan\v si\' c link $L\subset S^4$ has the following properties:
\begin{enumerate}
\item Any two components of $L$ form a trivial link, but three components do not;
\item The complement $M = S^4 \setminus L$ has an integral arithmetic hyperbolic metric; 
\item We may isotope $L$ so that it is preserved by a group of 640 isometries of $S^4$ that acts $k$-transitively on the components of $L$ for all $k=1,\ldots, 5$;
\item The double branched covering of $S^4$ ramified over $L$ has a $\matH^2 \times \matH^2$ metric;
\item The group $\pi_1(M)$ is generated by some meridians $a,b,c,d,e$ with relators
\begin{gather*}
[a,[c,d]], \quad [b,[d,e]], \quad [c,[e,a]], \quad [d,[a,b]], \quad [e,[b,c]], \\
[a,[b^{-1},e^{-1}]], \quad [b,[c^{-1},a^{-1}]], \quad [c,[d^{-1},b^{-1}]],
\quad [d,[e^{-1},c^{-1}]], \quad [e,[a^{-1},d^{-1}]].
\end{gather*}
The triple $a, [c,d], [b^{-1}, e^{-1}]$ generates a peripheral subgroup $\matZ^3$, and the same holds for the triples obtained by permuting cyclically $a,b,c,d,e$;
\item The Alexander ideal $I$ is generated by the monomials
$$(t_1-1)^{a_1}(t_2-1)^{a_2}(t_3-1)^{a_3}(t_4-1)^{a_4}(t_5-1)^{a_5}$$
such that $0\leq a_i \leq 4, a_1+\cdots + a_5 = 8$, at most one $a_i$ vanishes, and  
the condition $(a_{i-1},a_i,a_{i+1}), (a_{i-2},a_i,a_{i+2}) \neq (1,0,1) $ holds for all $i$, with indices interpreted modulo 5. Therefore the Alexander polynomial is $\Delta = 1$;
\item A cohomology class $\phi = (x_1,x_2,x_3,x_4,x_5) \in H^1(M, \matZ) = \matZ^5$ is represented by a perfect circle-valued Morse function $\Longleftrightarrow x_i \neq 0$ for all $i$; 
\item A longitudinal Dehn surgery on $L$ gives a 4-manifold $X$ with $\pi_1(X)=\matZ^5$;
\item The link $L$ can be put in perfect position.
\end{enumerate}
\end{teo}

We now make some comments on the terminology, and on each property (1)-(9). 

\subsubsection*{ Dehn filling and Dehn surgery} If $W$ is a compact $n$-manifold with boundary consisting of some $(n-1)$-tori, a \emph{Dehn filling} of $W$ is any compact $n$-manifold $X$ obtained by attaching via smooth maps one $T^{n-2} \times D^2$ to some (possibly all) boundary components of $W$. The core $(n-2)$-tori $T^{n-2}\times \{0\}$ will form a link of $(n-2)$-tori in the filled manifold $X$.

Conversely, let the interior of a $n$-manifold $X$ contain some disjoint $(n-2)$-tori with trivial normal bundle. A \emph{drilling} of $X$ consists of the removal of some small open tubular neighbourhoods of these $(n-2)$-tori. Finally, a \emph{Dehn surgery} along these $(n-2)$-tori consists of a drilling followed by a Dehn filling of the resulting boundary components. If $X$ was closed, the resulting surgered manifold also is.

\subsubsection*{{\rm(1)} Trivial sublink}
By definition a link of $k$ tori in $S^4$ is \emph{trivial} if it is isotopic to a split union of standard tori in some smooth 3-disc contained in $S^4$. The fundamental group of the complement of a trivial link of $k$ tori is easily seen to be the free group with $k$ generators. 

It is easily deduced from the film of Figures \ref{torus_all:fig} and \ref{tori:fig} that any two tori form a trivial link. In fact, the 2-transitivity in property (3) implies that it suffices to check this on a single pair of tori, for instance the black and red ones.

We note that this is also coherent with the presentation for $\pi_1(M)$ given in (5). We can recover the fundamental group of the complement of any collection of tori in $L$ by killing the meridians of the other tori: we can see easily that if we kill three meridians, say $a,b,c$, all the relators become trivial and we are left with a free group in the remaining generators $d,e$.

On the other hand, we can also use this method to check that three tori in $L$ do \emph{not} form a trivial link: using 3-transitivity from (3) we may suppose that the three tori have meridians $c,d,e$; if we kill the remaining meridians $a,b$, we obtain the group $\langle c,d,e | [d,[e^{-1},c^{-1}] \rangle$ that is not free. 

We have deduced property (1) from (3), (5) and by looking at Figures \ref{torus_all:fig} and \ref{tori:fig}.

\subsubsection*{{\rm (2)} Integral arithmetic hyperbolic complement}
This is in fact the origin of the Ivan\v si\' c link. \Ivan proved in \cite{I, I2} that the
orientable double cover of the non-orientable hyperbolic 4-manifold 1011 from the Ratcliffe -- Tschantz census \cite{RT} is a link complement in $S^4$, although the link was not explicitly determined there (the proof
consisted in showing that some Dehn filling of this manifold is diffeomorphic to $S^4$, and it is not easy in general -- not even in dimension 3 -- to recover an explicit picture of a link from its exterior: see \cite{DRO} for a recent breakthrough in that direction).
We will prove in Section \ref{RT:subsection} that this orientable double cover is in fact diffeomorphic to $M = S^4 \setminus L$.

All the manifolds from the Ratcliffe -- Tschantz census are \emph{integral arithmetic}, that is they are obtained as a quotient $\matH^4/\Gamma$ for some lattice $\Gamma < \SO^+(4,1) \cap \GL(5,\matZ)$, and they have $\chi =1$, see \cite{RT}. Thus the double cover $M$ is also integral arithmetic and has $\chi(M)=2$. By the Gauss -- Bonnet formula the volume is
$$\Vol(M) = \frac 43 \pi^2 \chi(M) = \frac 83 \pi^2.$$

Since $L$ consists of five tori, each with trivial normal bundle, the hyperbolic manifold $M$ has five cusps, each with a 3-torus cusp section. All the manifolds in \cite{RT} decompose into a single ideal 24-cell, which in turns decomposes into 16 copies of the right-angled polytope $P^4$. Therefore the double cover $M$ decomposes into two ideal 24-cells, and further into 32 copies of $P^4$. We will describe here a convenient decomposition of $M$ into these 32 copies of $P^4$ obtained by \emph{colouring} its facets in Section \ref{4:subsection}. We will deduce most of Theorem \ref{main:teo} from this description. In fact we will also prove that the manifold $M$ coincides with the manifold $M^4$ from \cite{IMM} and with the manifold $W$ from \cite{BM}.  

\subsubsection*{{\rm (3)} The 640 isometries}
Consider the isometries
\begin{equation*} 
(x_1, x_2, x_3, x_4, x_5) \longmapsto (\pm x_{\sigma(1)}, \pm x_{\sigma(2)}, \pm x_{\sigma(3)}, \pm x_{\sigma(4)}, \pm x_{\sigma(5)})
\end{equation*}
of $S^4$ where $\sigma\in S_5$ is any permutation contained in the order 20 group $H<S_5$ generated by the cycles $(12345)$ and $(2354)$, and signs $\pm$ are arbitrary. We get a group $G$ of $20 \times 32 = 640$ isometries and we will prove
in Section \ref{symmetries:subsection} that they preserve (a link isotopic to) the \Ivan link $L$ and act $k$-transitively (that is, transitively on the subsets of $k$ elements) on the components of $L$, for every $k$.

We also note that every isometry restricts to an isometry of the hyperbolic manifold $M=S^4 \setminus L$. This holds in general by Mostow -- Prasad rigidity, and it will be evident by construction in Section \ref{symmetries:subsection}. In fact each such isometry will preserve the tessellation of $M$ into 32 copies of the right-angled polytope $P^4$.

\subsubsection*{{\rm (4)} The double branched covering}
For every (not necessarily connected) closed orientable $(n-2)$-manifold $\Sigma \subset S^n$ there is a well-defined double branched covering over $S^n$ ramified along $\Sigma$. 

The double branched covering over $S^4$ ramified over the \Ivan link $L$ is a closed oriented 4-manifold $W$ with $\chi(W)=4$, and we could not expect it to be a flat 4-manifold as with the Borromean rings, since flat manifolds have $\chi = 0$. We are however  lucky enough to find that $W$ admits a geometry of type $\matH^2 \times \matH^2$. The manifold $W$ is not a product of surfaces because $b_1(W)=0$, see Corollary \ref{W:not:cor}, but it is finitely covered by a product of surfaces (likewise, the Hantsche -- Wendt flat 3-manifold that double covers $S^3$ branched along the Borromean rings $B$ is not a 3-torus and has $b_1=0$, but it is finitely covered by a 3-torus). This is proved in Section \ref{branched:subsection}.

We will in fact equip the pair $(S^4,L)$ with a geometric $\matH^2\times \matH^2$ orbifold structure that is singular along $L$ with cone angle $\pi$. The double branched cover $W$ will inherit automatically the structure of a $\matH^2\times \matH^2$ manifold.

This remarkable and quite unexpected fact implies in particular the following.

\begin{cor}
It is possible to obtain a closed manifold $W$ with geometry $\matH^2 \times \matH^2$ by Dehn filling the cusps of some manifold $N$ with geometry $\matH^4$. 
\end{cor}
\begin{proof}
The double branched covering $W$ over $L$ is a Dehn filling of a degree-two covering $N$ of the cusped hyperbolic manifold $M=S^4 \setminus L$. 
\end{proof}

This was previously unknown. Note that $\chi(W) = \chi(N)$, and this new phenomenon is (of course) coherent with the fact, proved in \cite{FM}, that the simplicial volume is non-increasing under Dehn filling: indeed we have
\begin{align*}
\|N\|  & = \frac{\Vol(N)}{v_4} = \frac {4\pi^2}{3v_4}\chi(N) \sim 48.94 \cdot \chi(N),\\
\|W\| & = 6 \cdot \chi(W) = 6 \cdot \chi(N).
\end{align*}
where $v_4 \sim 0.26889 \ldots$ is the volume of the regular ideal hyperbolic 4-simplex \cite{HM}, and the second equality was proved in \cite{B}. We get $\|W\| < \|N\|$ as expected.

We also note again that it is \emph{not} possible to obtain a flat 4-manifold by Dehn filling a hyperbolic manifold since they have different Euler characteristic. 

\subsubsection*{{\rm (5)} Fundamental group}
The generators $a,b,c,d,e$ are some \emph{meridians}, that is some curves that encircle the five tori in $L$. The given presentation is calculated in Section \ref{pi:subsection}.
We note that the relators are remarkably simple and symmetric; similarly as with the Borromean rings, the elements 
$$a, \quad [c,d], \quad [b^{-1}, e^{-1}]$$ 
represent a basis of the peripheral group $\matZ^3$ that is the boundary of a regular neighbourhood of one torus of the link. The same holds for each torus after permuting cyclically the generators $a,b,c,d,e$.
We note that by Alexander duality 
$$b_1(M) = 5, \quad b_2(M) = 10, \quad b_3(M) = 4.$$

We deduce that the presentation of $\pi_1(M)$ has the minimum possible number 5 of generators and 10 of relators. 


\subsubsection*{{\rm (6)} The Alexander ideal}
The Alexander ideal $I$, defined by Fox \cite{F}, lies in 
$$\matZ[H^1(M,\matZ)] = \matZ[t_1^{\pm 1}, \ldots, t_5^{\pm 1}].$$ 
It has been determined here from the presentation of $\pi_1(M)$ using Fox calculus \cite{F} and a computer program. We explain this in Section \ref{Fox:subsection}.

The Alexander polynomial $\Delta$ is the greatest common divisor of the Alexander ideal, well-defined up to multiplication with a product of monomials $\pm t_i^{\pm 1}$, see also \cite{McM}. When $M$ is a 3-manifold this is a famous and interesting polynomial, while for more general spaces it may happen that the Alexander polynomial is just $\Delta=1$, and this is what we get here. 

Even when the Alexander polynomial $\Delta$ carries no useful information, the Alexander ideal may be used to compute the first Betti number of the infinite cyclic coverings, as explained in \cite{F, McM}. We will employ it to prove Theorem \ref{Betti:teo} below. 


\subsubsection*{{\rm (7)} Perfect circle-valued Morse functions}
Every primitive integral cohomology class $\phi = (x_1,x_2,x_3,x_4,x_5) \in H^1(M,\matZ) = \matZ^5$ may be interpreted as a surjective homomorphism $\phi \colon \pi_1(M) \to \matZ$ and it is natural to wonder whether $\phi$ can be represented by a Morse function $M\to S^1$ with the minimum number $\chi(M) = 2$ of singular points (necessarily of index 2, see \cite{BM}). Such a Morse function is called \emph{perfect}, and it should be interpreted as the natural generalisation of a fibration when $\chi(M) \neq 0$.

As already mentioned, the manifold $M$ coincides with the manifold $W$ considered in \cite{BM}, so we already know when $\phi$ is represented by a perfect circle-valued Morse function: by \cite[Theorem 4]{BM} this holds precisely when $x_i \neq 0$ for all $i$, so there is nothing to prove here.

\subsubsection*{{\rm (8)} Longitudinal Dehn surgery}
Recall that a Dehn surgery along $L$ consists of removing the tubular neighbourhood $\nu = D^2 \times T$ of each torus $T$ and gluing it back via some diffeomorphism $\varphi$ of $\partial \nu$. 
We say that the Dehn surgery is \emph{longitudinal} if $\varphi$ sends the meridian $S^1 \times p$ to a curve in $\partial \nu$ that is null-homologous in the complement $S^4 \setminus T$. As opposite to dimension three, there are infinitely many non-homotopic curves that fulfill this requirement, and there does not seem to be a canonical way to choose one. We get infinitely many potentially different longitudinal Dehn surgeries along each component of $L$.

Every longitudinal Dehn surgery along $L$ produces a closed 4-manifold $X$ with $H_1(X) = \matZ^5$.
We prove in Section \ref{Dehn:subsection} that there is a particular longitudinal Dehn surgery where we also get $\pi_1(X) = \matZ^5$. As opposite to the Borromean rings, such a manifold $X$ cannot be flat, and neither it can be aspherical, because its fundamental group has cohomological dimension 5.

We recall from \cite{FM0} that every sufficiently complicated Dehn surgery on a link of tori (in any closed 4-manifold) with hyperbolic complement gives rise to a manifold that is locally CAT(0), and hence in particular aspherical: this implies that most longitudinal surgeries on $L$ yield aspherical manifolds. Therefore the surgery yielding $\pi_1(X^5) = \matZ^5$ should be interpreted as exceptional.

\subsubsection*{{\rm (9)} Perfect position}
We say that a submanifold $\Sigma\subset S^n$ is in \emph{perfect position} if the height function $x_{n+1}\colon S^n \to [-1,1]$ restricts to a perfect Morse function on $\Sigma$, that is one with exactly $b_i(\Sigma)$ points of index $i$ for all $i$. The Borromean rings $B\subset S^3$ in Figure \ref{Borromean:fig} are in perfect position. 

We prove in Section \ref{perfect:subsection} that the \Ivan link $L\subset S^4$ can be put in perfect position. This means that the height function has, on each torus in $L$, one minimum, two saddle points, and one maximum. We note that the pictures in Figures \ref{torus_all:fig} and \ref{tori:fig} do \emph{not} display $L$ in perfect position, since the height function on each torus has the double of the critical points requested for every index. We prove that $L$ can be put in perfect position, but we were not able to show this directly with a picture. 

\vspace{.4 cm}
We end the Introduction with some further comments, analysing some consequences of our work, and asking some natural open questions.

\subsection*{Betti numbers of the infinite cyclic coverings}
Let $M=S^4\setminus L$ be the complement of the \Ivan link $L\subset S^4$.
Every primitive integral cohomology class $\phi = (x_1,x_2,x_3,x_4,x_5) \in H^1(M,\matZ) = \matZ^5$ determines a surjective homomorphism $\phi\colon \pi_1(M) \to \matZ$ and hence an infinite cyclic cover $\tilde M$ determined by $\ker \phi$. It is natural to wonder how the Betti numbers $b_i(\tilde M)$ are depending on $\phi$.

\begin{teo} \label{Betti:teo}
The following hold:
\begin{enumerate}
\item $b_1(\tilde M) = 8 + d(x_1,x_2,x_3,x_4,x_5)$ for some computable finite $d\geq 0$ if at most one $x_i$ vanishes, and $b_1(\tilde M) = \infty $ otherwise; 
\item $b_2(\tilde M) = \infty$,
\item $b_3(\tilde M) = 0$ if $x_i \neq 0$ for all $i$ and $b_3(\tilde M) = \infty$ otherwise.
\end{enumerate}
The function $d$ is quite irregular. In particular when $x_i \neq 0$ for all $i$ we have:
\begin{itemize}
\item $d=0$ if and only if $(x_1,x_2,x_3,x_4,x_5)$ are pairwise coprime;
\item $d=p-1$ if $(x_1,x_2,x_3,x_4,x_5) = (p,p,1,1,1)$ for some positive prime $p$.
\end{itemize}
\end{teo}

This is proved in Section \ref{infinite:subsection}. To determine $b_1(\tilde M)$ we use a result of Milnor \cite{McM, Mil} that allows us to derive this number from the generators of the Alexander ideal, found in Theorem \ref{main:teo}-(6). In Section \ref{infinite:subsection} we also prove that in fact $b_2(\tilde M) = \infty$ for every infinite cyclic cover $\tilde M$ of any finite-volume hyperbolic 4-manifold $M$. 

The behaviour of $b_1(\tilde M)$ is quite irregular, in stark contrast with the Borromean rings complement, and in fact with any 3-manifold. It is an observation of McMullen \cite{McM} that on any compact orientable 3-manifold $N$ the function $\phi \mapsto b_1(\ker \phi)$ is \emph{affine} on primitive classes inside the cone of every open face of the Alexander polyhedron: for instance for the Borromean rings complement all the faces are fibered and when $x_i \neq 0$ for all $i$ we have 
$$b_1(x_1,x_2,x_3) = 1+|x_1|+|x_2|+|x_3|.$$ 

\subsection*{No essential spheres, tori, discs, annuli}
The hyperbolicity of the complement $M=S^4\setminus L$ of the \Ivan link $L$ has some immediate consequence on its topology, that are familiar in dimension 3 but which in fact hold in any dimension. First, the manifold $M$ is aspherical. Second, it does not contain any essential \emph{sphere}, \emph{torus}, \emph{disc}, or \emph{annulus}. That is, any map $f\colon X \to \bar M$ from one such surface $X$ to the compact version $\bar M$ of $M$  such that $f^{-1}(\partial \bar M) = \partial X$ is either not $\pi_1$-injective, or homotopic (rel to the boundary) to a map in a collar of $\partial \bar M$. (Here $\bar M$ is the compact manifold with boundary whose interior is $M$.)

These conditions are quite restrictive, even more in dimension $n=4$ since the boundary 3-tori contain many closed curves that could potentially bound essential discs or annuli. The black torus in Figure \ref{torus_all:fig} separates the equator into two solid tori, each containing a \emph{chain link} with four components: the chain links are there to prevent the black torus to bound an essential disc in the solid torus, or to cobound an essential annulus with some other torus. In fact the complement of the chain link in each solid torus is a hyperbolic 3-manifold with five cusps totally geodesically embedded in the hyperbolic 4-manifold $M=S^4\setminus L$ (because it is a fixed points set of the symmetry $(x_1,x_2,x_3,x_4,x_5) \mapsto (x_1,x_2,x_3,x_4,-x_5)$, which is an isometry of $M$). If we remove one of the five tori we can see immediately from the pictures the appearance of some essential discs or annuli that will prevent the complement of the remaining four tori from being hyperbolic.

We have just remarked that $M$ contains some cusped totally geodesic hyperbolic 3-manifold; we note that by a result of Chu and Reid it contains no closed one \cite{CR}.

\subsection*{Hyperbolic links of tori in other closed 4-manifolds}
A link of tori in a closed 4-manifold $X^4$ is \emph{hyperbolic} if its complement admits a complete finite-volume hyperbolic structure. In dimension 3 every closed 3-manifold contains many hyperbolic links, but in dimension 4 hyperbolic links of tori should be in principle more sporadic: the complement of a hyperbolic link of tori in $X^4$ has the same (necessarily positive) Euler characteristic of $X^4$, and there are only finitely many cusped hyperbolic 4-manifolds with any fixed positive Euler characteristic \cite{Wang}. So a fixed closed 4-manifold $X^4$ contains only finitely many cusped hyperbolic 4-manifolds as link complements (this does not imply that it contains finitely many hyperbolic links up to isotopy, since infinitely many non-isotopic links may share the same complement). 

In particular we do not know which closed 4-manifolds with positive Euler characteristic contain at least one hyperbolic link. Among the few examples known we have $S^4$ \cite{I, I2} and some manifolds homeomorphic to $\#_{2k} (S^2 \times S^2)$ \cite{S}.
The techniques used in this paper allows us to find two more natural examples.

\begin{teo} \label{Lagrangian:teo}
The real projective space $\matRP^4$ and the product $\Sigma \times \Sigma$ of two surfaces $\Sigma$ of genus two contain each a hyperbolic link of tori.

The hyperbolic link in $\matRP^4$ is the image of $L\subset S^4$ along the antipodal map. It consists of 5 tori and its complement is the non-orientable hyperbolic manifold 1011 from the Ratcliffe -- Tschantz census \cite{RT}. The cusp sections are flat torus bundle over $S^1$ with monodromy $\matr {-1}001$.

The hyperbolic link in $\Sigma \times \Sigma$ consists of the following 9 Lagrangian tori:
$$\alpha \times \alpha, \quad \beta \times \gamma_i, \quad \gamma_i \times \epsilon, \quad \delta_i \times \beta, \quad \epsilon \times \delta_i$$
where the curves $\alpha, \beta, \gamma_1, \gamma_2, \delta_1, \delta_2, \epsilon$ are in Figure \ref{genus_two_greek:fig}, and $i\in \{1,2\}$ varies. 
\end{teo}

These are the first examples of hyperbolic links of tori in $\matRP^4$ and in a product of surfaces. The theorem is proved in Sections \ref{RT:subsection} and \ref{Lagrangian:subsection}.

\begin{figure}
\centering
\includegraphics[width=8 cm]{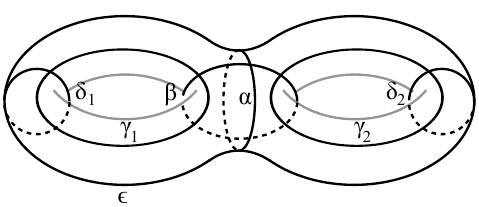}
\caption{A genus two surface and the simple closed curves $\alpha, \beta, \gamma_1, \gamma_2, \delta_1, \delta_2, \epsilon$. Note that these curves cut the surface into 8 pentagons.}
\label{genus_two_greek:fig}
\end{figure}

\subsection*{Further questions}
We start with a natural problem.

\begin{probl}
Produce more examples of links of tori in $S^4$ with hyperbolic complement. Is there a link that consists of a single torus? 
\end{probl}

Ivan\v si\' c, Ratcliffe and Tschantz \cite{IRT} have produced more examples of hyperbolic links of tori and Klein bottles in a smooth 4-manifold that is homeomorphic to $S^4$, and most likely diffeomorphic to it. More generally, the following seems open.

\begin{quest}
Is there a knotted torus in $S^4$ with aspherical complement?
\end{quest}

We recall from \cite{FM0} that a sufficiently complicated Dehn surgery along a link with hyperbolic complement is locally CAT(0), and hence in particular aspherical with universal cover $\matR^4$.

\begin{probl}
Classify all the \emph{exceptional} Dehn surgeries on the \Ivan link $L\subset S^4$, that is those that do not produce a CAT(0) space. This has been done for the Borromean rings in \cite{Mar7}.
\end{probl}

The description of a hyperbolic manifold as a complement of some Lagrangian tori in a product of surfaces is explicit in Theorem \ref{Lagrangian:teo}, and it is natural to ask whether it can be generalised.

\begin{probl}
Determine which collections of disjoint Lagrangian tori in a product of surfaces has a hyperbolic complement.
\end{probl}

Not much seems known in dimension $n>4$. The paper \cite{IMM} contains a notable sequence of cusped hyperbolic manifolds $M^3,\ldots, M^8$ of dimension $n=3,\ldots, 8$, where $M^3$ and $M^4$ are precisely the complements of the Borromean rings $B$ in $S^3$ and of the \Ivan link $L$ in $S^4$. It is tempting to pursue our investigation with $M^5$, which has also the delightful property of fibering over the circle \cite{IMM2}. However, the Betti numbers of $M^5, \ldots, M^8$ calculated in \cite{IMM} show that they cannot be link complements in $S^5, \ldots, S^8$, because they would contradict Alexander duality.

\begin{quest}
Is there a cusped hyperbolic 5-manifold that is a complement of 3-tori in $S^5$? More generally, is there a cusped hyperbolic $n$-manifold with $n\geq 5$ that is a complement of $(n-2)$-tori in a closed $n$-manifold of some simple kind?
\end{quest}

The following question also seems open.

\begin{quest}
Is there a link of $(n-2)$ tori in $S^n$ for some $n>4$ with aspherical complement?
\end{quest}

Note that the complement of an unknotted $S^{n-2} \subset S^n$ is homotopically equivalent to $S^1$ and hence aspherical, while the complement of an unknotted $T^{n-2} \subset S^n$ is not aspherical when $n>3$, since its complement has $\pi_1=\matZ$ and $b_2 \neq 0$.

\subsection*{Structure of the paper}
The rest of the paper is devoted to the proofs of Theorems \ref{main:teo}, \ref{Betti:teo}, and \ref{Lagrangian:teo}. In Section \ref{structure:section} we prove Theorem \ref{main:teo}-(2), first by constructing $M$, then by showing that it has a Dehn filling diffeomorphic to $S^4$ with core tori $L$. In Section \ref{geometric:section}
we prove Theorem \ref{main:teo}-(3, 4, 9). In Section \ref{fundamental:section}
we prove Theorem \ref{main:teo}-(5, 6, 8), then we use the Alexander ideal to prove Theorem \ref{Betti:teo}. Finally, in Section \ref{link:section}
we prove Theorem \ref{Lagrangian:teo} and show that $M$ is indeed the double cover of the non-orientable hyperbolic 4-manifold 1011 from the Ratcliffe -- Tschantz census \cite{RT} considered by \Ivan \cite{I, I2}.

\subsection*{Acknowledgements}
We warmly thank Alan Reid for discussions and for suggesting that the manifold $M^4$ in \cite{IMM} might coincide with the link complement in $S^4$ determined by Ivan\v si\' c. This observation has motivated this research.

\section{The hyperbolic structure} \label{structure:section}
We show here that the complement $M=S^4\setminus L$ of the \Ivan link is hyperbolic, thus proving Theorem \ref{main:teo}-(2). In fact, 
we prove that the complement $M$ is diffeomorphic to the cusped hyperbolic manifold defined in \cite{BM, IMM}
respectively with the names $W$ and $M^4$.

The hyperbolic manifold $M^4$ is constructed in \cite{IMM} by \emph{colouring} a right-angled hyperbolic polytope $P^4$. Our strategy is to \emph{Dehn fill} $P^4$ as explained in \cite{MR}. This operation will determine a Dehn filling of $M^4$, it will allow us to prove that this Dehn filling is diffeomorphic to $S^4$, and to show that its core tori form the \Ivan link $L$. Therefore we will deduce that $M^4 \cong M = S^4\setminus L$.




\subsection{The Borromean rings} \label{3:subsection}
As a warm up, we first consider the Borromean rings; much of the 4-dimensional construction will be similar, and this also explains the analogies between the Borromean rings $B\subset S^3$ and the \Ivan link $L\subset S^4$.

\subsubsection{Colouring}
We start with the right-angled polyhedron $P^3\subset \matH^3$ shown in Figure \ref{bipyramid:fig}. The polyhedron is a bipyramid with three ideal vertices and two real ones, and is combinatorially dual to a triangular prism. A \emph{colouring} on a polyhedron is the assignment of a number (called a \emph{colour}) at each face, so that any two faces that meet along an edge have different colours. We colour the faces of $P^3$ with the colours $1,2,3$ as shown on the dual graph in Figure \ref{bipyramid:fig}-(right). This is the unique colouring of $P^3$ with 3 colors, up to isometries of $P^3$ and permutations of the colors. 

As explained in \cite{IMM}, this colouring produces a hyperbolic 3-manifold $M^3$ tessellated into $2^3=8$ copies of $P^3$, with 3 cusps. The procedure to construct $M^3$ goes as follows: we take $2^3$ identical copies $P_v^3$ of $P^3$ parametrised by $v \in (\matZ/2\matZ)^3$, and we identify every face of $P_v^3$ coloured with $i$ with the corresponding face of $P_{v+e_i}^3$ via the identity map. We refer to \cite{IMM} for more details.

\begin{figure}
\centering
\includegraphics[width=5 cm]{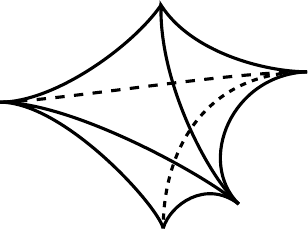}
\quad
\includegraphics[width=4 cm]{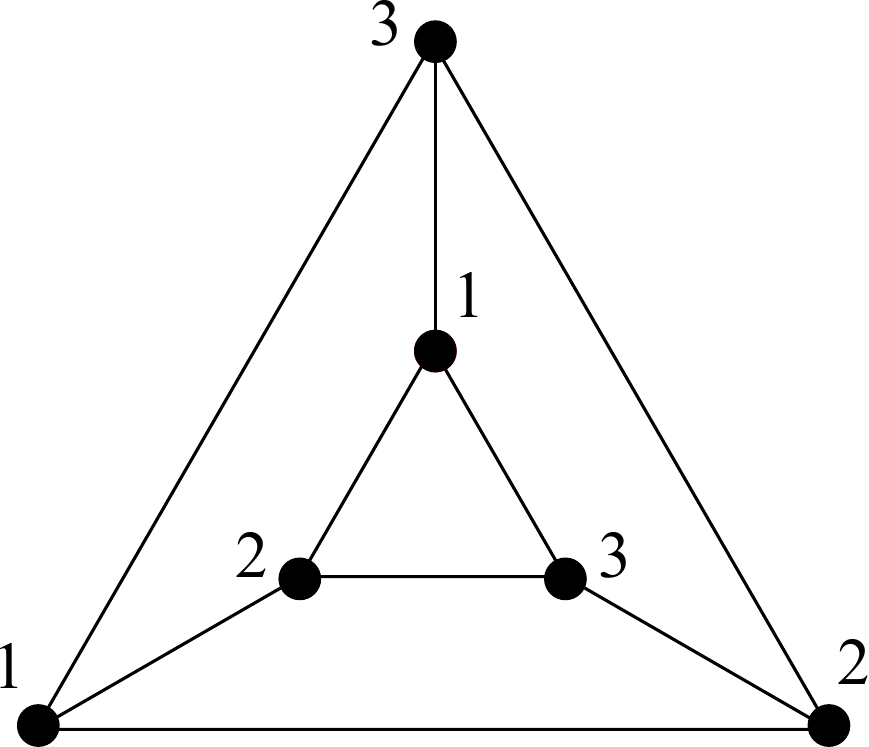}
\caption{The right-angled bipyramid $P^3\subset \matH^3$ has two real vertices with valence 3 and three ideal vertices with valence 4 (left). We colour its faces with $1,2,3$ as shown in the dual graph, that is the 1-skeleton of a triangular prism (right).}\label{bipyramid:fig}
\end{figure}

\subsubsection{Dehn filling} \label{Dehn:subsubsection}
We show here that $M^3$ is actually the Borromean rings complement. To this purpose
we \emph{Dehn fill} the polyhedron $P^3$ as explained in \cite{MR}. The Dehn filling procedure consists of substituting each ideal vertex of $P^3$ with a segment: it can be seen as a two-step operation in which we first truncate the ideal vertex and then shrink the resulting new square to a segment, as in Figure \ref{Dehn_fill:fig}. 

\begin{figure}
\centering
\includegraphics[width=11 cm]{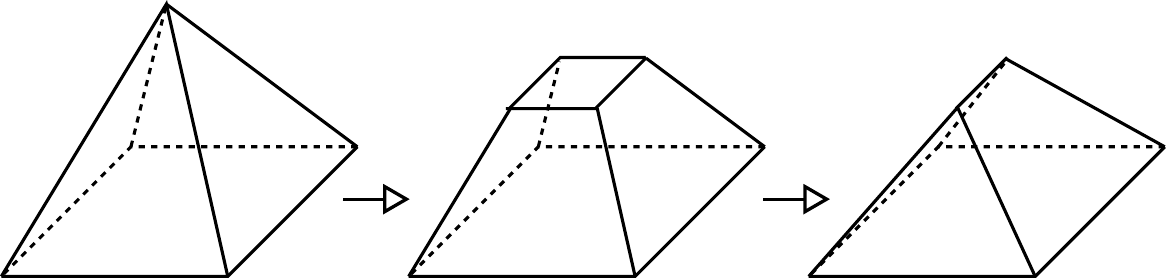}
\caption{Dehn filling an ideal vertex: we first truncate and then shrink the square to a segment (in general dimension $n$, we shrink the $(n-1)$-cube to a $(n-2)$-cube).}
\label{Dehn_fill:fig}
\end{figure}

At every ideal vertex we may perform two different Dehn fillings, depending on which direction we choose to collapse the square onto a segment. Since $P^3$ has three ideal vertices we get $2^3=8$ different choices overall. Two choices produce a combinatorial cube, shown in Figure \ref{filledP3:fig}, while the six other choices produce different kinds of combinatorial hexahedra.

\begin{figure}
\centering
\includegraphics[width=12 cm]{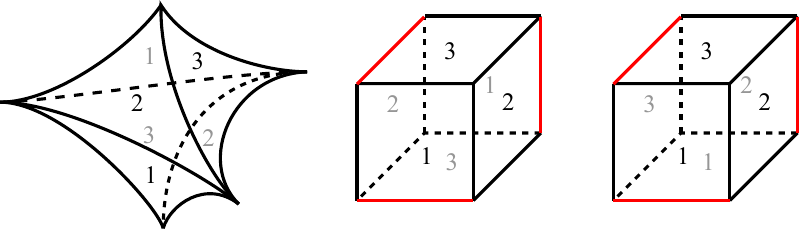}
\caption{The coloured $P^3$ (left) may be Dehn filled in 8 different ways, producing different coloured polyhedra. A Dehn filling substitutes each ideal vertex with a new edge, drawn here in red. Two choices produce a cube (center and right), while the other six choices produce different kinds of hexahedra. The colours on the back faces are drawn in gray. Note that the two resulting cubes inherit substantially different colourings.}
\label{filledP3:fig}
\end{figure}

The Dehn filled polyhedron inherits a colouring from that of $P^3$. However, it may be that the inherited colouring is not strictly speaking a colouring, because two faces that have a red edge in common may share the same colour: this is the case in Figure \ref{filledP3:fig}-(right). 
This apparently annoying phenomenon may occur since some faces that were not adjacent along an edge in $P^3$ may become so after the Dehn filling. In that case we simply remove the red edge separating the two similarly coloured faces (we call this operation a \emph{smoothing} of an edge) and unify the two faces into a unique face. Smoothing solves the colouring problem, at the price of getting a ``polyhedron'' that may not quite be a polyhedron in a strict sense: it is a topological disc with corners, that we still call a polyhedron for simplicity. 

\begin{figure}
\centering
\includegraphics[width=11 cm]{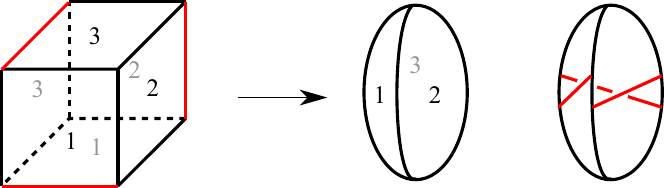}
\caption{If we smoothen the red edges we get a coloured ``polyhedron'' with only 2 vertices and 3 bigonal faces. The polyhedron is the suspension of a triangle. It is useful to keep track of the red edges in the smoothed polyhedron as shown on the right.}
\label{smoothen:fig}
\end{figure}

For each of the 8 possible choices of Dehn fillings on $P^3$ we get a coloured polyhedron, possibly after smoothing some red edges separating faces sharing the same colour. We can now apply the colouring procedure to the resulting polyhedron, and this produces a closed 3-manifold $X^3$. By construction $X^3$ is obtained from $M^3$ topologically by adding some red circles $L$ (the copies of the red edges), hence $X^3$ is indeed a Dehn filling of $M^3$, and by taking track of the red edges we can identify the link $L\subset X^3$ whose complement is $M^3$.

Summing up, we have 8 choices of Dehn fillings for $P^3$, and each will give a Dehn filling $X^3$ of $M^3$ that is also described from a coloured polyhedron. Let us analyze the two cubic Dehn fillings of Figure \ref{filledP3:fig}. We now show that these give $X^3 =T^3$ and $X^3 = S^3$ respectively.

In the central cube of Figure \ref{filledP3:fig}, opposite faces share the same colour. This is a standard configuration studied for instance in \cite{IMM}, and it indeed produces a flat 3-torus $X^3=T^3$ tessellated into 8 such cubes. The red edges will form a link $L$ that consists of three pairwise disjoint geodesics (each made of two red edges) in the three different coordinate directions. The hyperbolic manifold $M^3$ is homeomorphic to $T^3 \setminus L$ by construction. It is well known that such a complement is homeomorphic to the Borromean ring complement, and we will prove it directly in the next lines.

In the right cube of Figure \ref{filledP3:fig}-(right) some smoothing is needed because there are adjacent faces sharing the same colour. The smoothed Dehn filled polyhedron is shown in Figure \ref{smoothen:fig}. It has two vertices and three bigonal faces, that is it is the suspension of a triangle. It cannot be realized as a polyhedron in $\matR^3$, but most interestingly it can be realized geometrically as a \emph{spherical} right-angled polyhedron: this is the suspension of the right-angled spherical triangle, that is the polyhedron in $S^3$ delimited by three pairwise orthogonal planes.
If we perform the colouring construction, we get $S^3$ tessellated into 8 such polyhedra, and the red edges shown in Figure \ref{smoothen:fig}-(right) are easily shown to form the Borromean rings $B\subset S^3$. So indeed $M^3$ is the Borromean rings complement.

\begin{prop} \label{B:isom:prop}
The Borromean rings $B$ in the form just described are preserved by the following 48 isometries of $S^3$
$$(x_1,x_2,x_3,x_4) \longmapsto (\pm x_{\sigma(1)}, \pm x_{\sigma(2)}, \pm x_{\sigma(3)}, {\rm sgn}(\sigma) x_{\sigma(4)})$$
where the signs $\pm$ are arbitrarly and $\sigma \in S_3$ is any permutation.
\end{prop}
\begin{proof}
The Dehn filled $P^3$ is the suspension of the right-angled triangle, and we let $x_4$ be the direction of the suspension. The sphere $S^3$ is tessellated into 8 copies of it. By construction $L$ is symmetric by reflection along the spheres $x_1=0, x_2=0$ and $x_3=0$.

Moreover, for every $\sigma \in S_3$ there is an isometry of the Dehn filled $P^3$, \emph{i.e.}~the suspension of the right-angled triangle of Figure \ref{smoothen:fig}-(right), that acts like $\sigma$ on the coloured faces and preserves the red edges. This isometry inverts the two vertices if ${\rm sgn}(\sigma) = -1$. It extends to an isometry of $S^3$ that preserves $L$ and acts as 
$$(x_1,x_2,x_3,x_4) \longmapsto (x_{\sigma(1)}, x_{\sigma(2)}, x_{\sigma(3)}, {\rm sgn}(\sigma) x_{\sigma(4)}).$$
The proof is complete.
\end{proof}

\subsubsection{Double branched covering} \label{double:B:subsubsection}
We can also use $P^3$ to understand the geometry of the double branched covering $W^3$ over $S^3$ ramified over the Borromean rings. More precisely, we can use the Dehn filling construction to assign a $\pi$-orbifold Euclidean structure to the pair $(S^3,B)$. The double branched covering $W^3$ will inherit a flat structure. It is well known that $W^3$ is the Hantsche--Wendt flat 3-manifold \cite{Z}.

To do so, it suffices to apply the colouring method to the \emph{Euclidean} cube of Figure \ref{smoothen:fig}-(left) instead of the \emph{spherical} right-angled polyhedron of Figure \ref{smoothen:fig}-(right). The colouring method will produce topologically the same 3-sphere as above, but with a flat structure with cone angle $\pi$ along $B$ (adjacent faces along red edges share the same colour, hence only 2 cubes will be adjacent to the red edges instead of 4 in the other edges, thus forming a cone angle $\pi$). 


Summing up, we have discovered that (i) $M^3=S^3 \setminus B$, (ii) $M^3 = T^3 \setminus L$, (iii) the double branched covering over $(S^3,B)$ is flat, only by Dehn filling $P^3$. Our aim is now to do something similar with the right-angled polytope $P^4 \subset \matH^4$.

\subsection{The \Ivan link} \label{4:subsection} 
We mimic the discussion of Section \ref{3:subsection} one dimension higher. We first recall the construction of the manifold $M^4$ from the polytope $P^4$ of \cite{IMM}, then we Dehn fill $P^4$, we prove that the corresponding Dehn filling of $M^4$ is $S^4$, and that the core tori form the \Ivan link $L$.

\subsubsection{The hyperbolic manifold $M^4$}
The right-angled polytope $P^4 \subset \matH^4$ was first introduced in \cite{PV, RT}. It has 10 facets, each isometric to $P^3$, five real vertices, and five ideal vertices.
The polytope $P^4$ is combinatorially dual to the 4-dimensional rectified simplex, that is the convex hull in $\matR^5$ of the 10 vertices obtained by permuting the coordinates of $(0,0,0,1,1)$. Two such vertices are connected by an edge of the rectified simplex if and only if they differ by only two coordinates. A convenient projection of the 1-skeleton taken from \cite{IMM} is shown in Figure \ref{P4:fig}.

\begin{figure}
 \begin{center}
  \includegraphics[width = 12 cm]{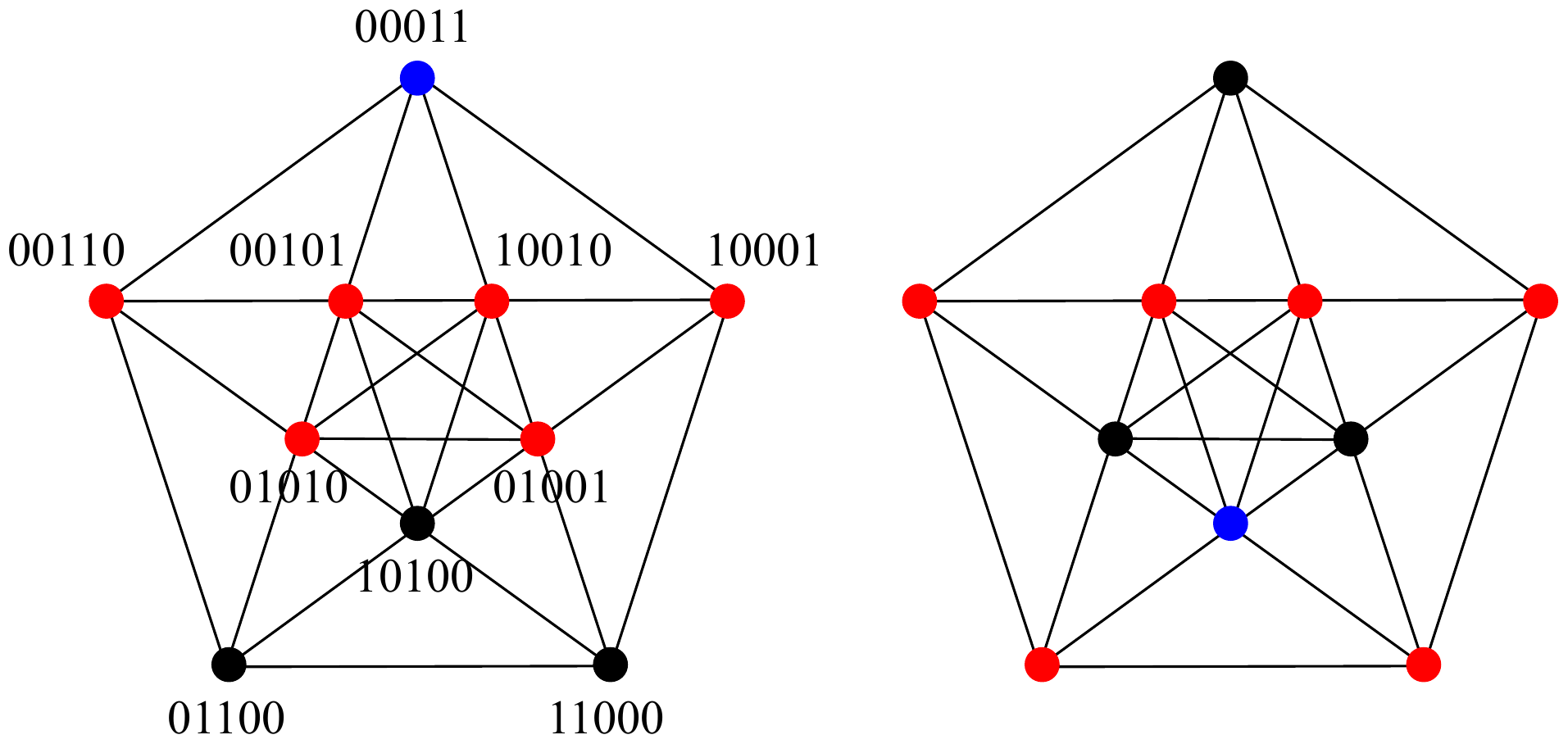}
 \end{center}
 \caption{A planar orthogonal projection of the 1-skeleton of the rectified simplex. Some edges are superposed along the projection, so two vertices that are connected by an edge on the plane projection may not be so in the rectified simplex. To clarify this ambiguity we have chosen a blue vertex and painted in red the 6 vertices adjacent to it, in two cases (all the other cases are obtained by rotation).}
 \label{P4:fig}
\end{figure}

\begin{figure}
 \begin{center}
  \includegraphics[width = 5 cm]{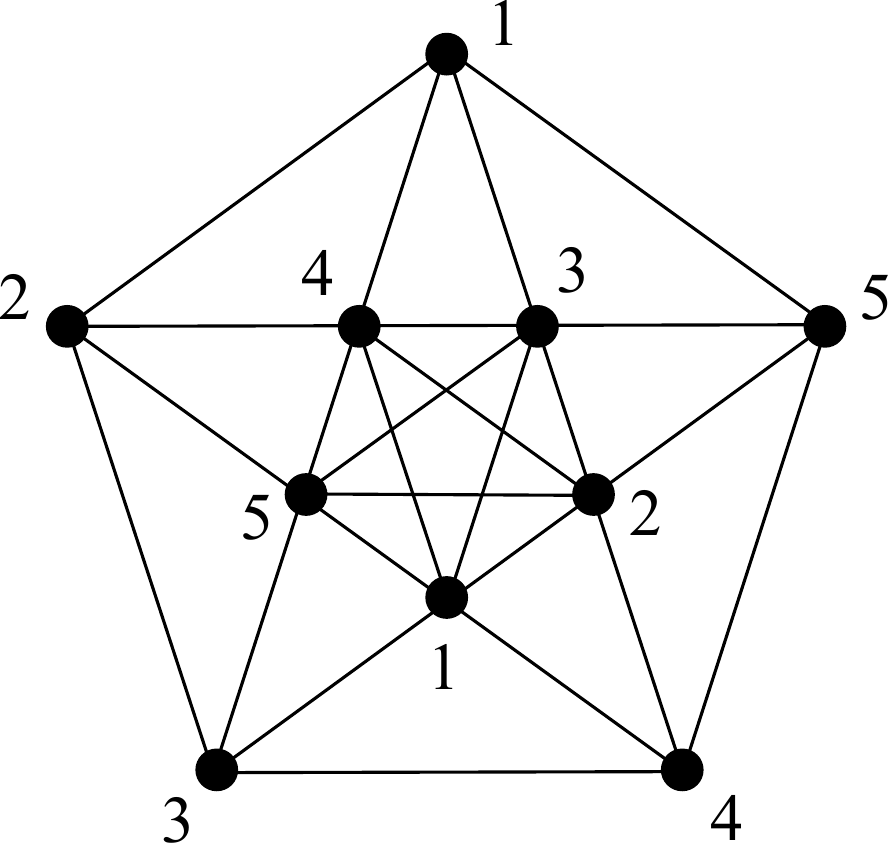}
 \end{center}
 \caption{A 5-colouring of the 1-skeleton of the rectified simplex and hence of $P^4$.}
 \label{P4_colour:fig}
\end{figure}

\begin{figure}
 \begin{center}
  \includegraphics[width = 3.5 cm]{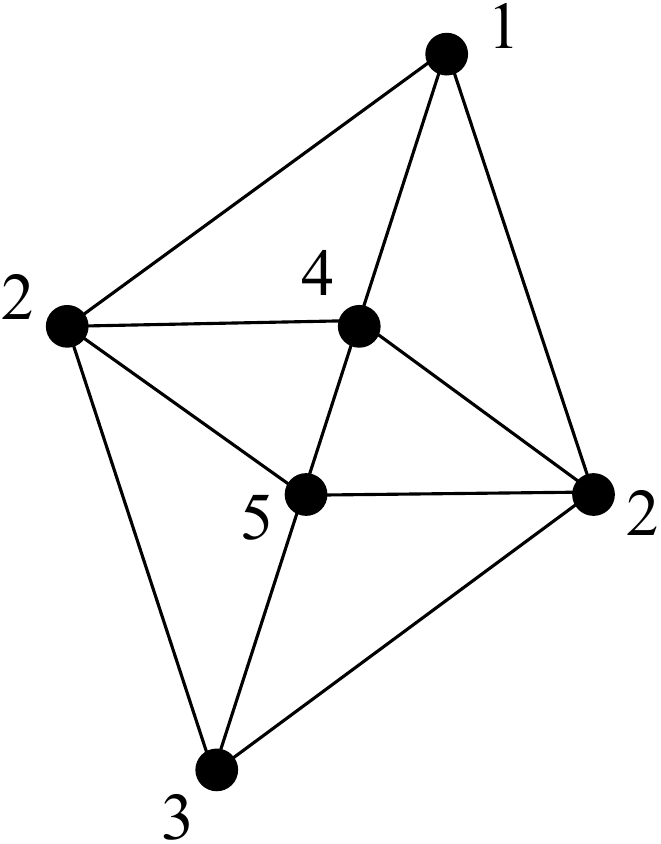}
 \end{center}
 \caption{An octahedral facet of the rectified simplex. This is a subgraph of the 1-skeleton in Figure \ref{P4_colour:fig}. Some edges are superposed. It inherits a 5-colouring where one pair of opposite vertices share the same colour 2.}
 \label{P4_facet:fig}
\end{figure}

Following \cite{IMM} we colour $P^4$ as in Figure \ref{P4_colour:fig} (colouring the vertices of the dual rectified simplex is analogous to colouring the facets of $P^4$).
This produces a hyperbolic 4-manifold $M^4$, tessellated into $2^5 = 32$ copies of $P^4$. We have $\chi(M^4) = 32/16 = 2$. As shown in \cite{IMM} the manifold $M^4$ has five cusps, one above each ideal vertex of $P^4$, and its Betti numbers are 
$$b_0(M^4) =1, \quad b_1(M^4) =5, \quad b_2(M^4) =10, \quad b_3(M^4) =4.$$

By Alexander duality, these are precisely the Betti numbers required to be a complement of five tori in $S^4$. 

\subsubsection{Dehn filling} \label{Dehn:4:subsubsection}
The rectified simplex has five octahedral and five tetrahedral facets, that are combinatorially dual to the five ideal and five real vertices of $P^4$. Figure \ref{P4_facet:fig} shows an octahedral facet, and by looking at it we note the following fact, that will be crucial for us: 

\begin{center}
\emph{The 6 vertices have 5 distinct colours, with one pair sharing the same colour.}
\end{center}

The pair sharing the same colour consists of two opposite vertices with colour 2. Since everything in Figure \ref{P4_colour:fig} is equivariant under $2\pi/5$ rotation, the other four octahedral facets of the rectified simplex also have this property. This has the following effect on $P^4$: 
\begin{center}
\emph{At every ideal vertex $v$ of $P^4$, the cube link of $v$ inherits a colouring from $P^4$ \\ with precisely one pair of opposite faces sharing the same colour}. 
\end{center}

We now Dehn fill $P^4$ following \cite{MR}. As in Figure \ref{Dehn_fill:fig}, this consists of truncating the ideal vertices, and then shrinking each new cubic face to a red square. A cube can be shrinked to a square in 3 possible ways, depending on the choice of two opposite faces: we choose the two opposite faces that share the same colour. This is a choice analogous to the one that we made for $P^3$ in Section \ref{Dehn:subsubsection}
when we picked the Dehn filling of Figure \ref{filledP3:fig}-(right). The Dehn filled polytope $\bar P^4$ is obtained from $P^4$ by substituting 5 ideal vertices with 5 red squares.

Similarly to what we have seen in Figure \ref{smoothen:fig}, two facets of $P^4$ that share the same colour are opposite with respect to some ideal vertex $v$ and therefore become incident in $\bar P^4$ along a new red square ridge: therefore in order to get an honest colouring we must smoothen and glue each such pair to a single facet. Let $\Delta^4$ be the polytope that we obtain from $\bar P^4$ by smoothing the red squares. Recall that \emph{a priori} the ``polytope'' $\Delta^4$ is only a disc with corners, not necessarily realizable as a polytope in some geometry. In fact it turns out to be an honest polytope.

\begin{prop}
The polytope $\Delta^4$ is a 4-simplex.
\end{prop}
\begin{proof}
The polytope $P^4$ has five real vertices, five ideal ones, and ten facets. The Dehn filling kills the ideal vertices and then the smoothing glues the facets in pair, so $\Delta^4$ has five vertices and five facets: this suggests that it should be a 4-simplex, but we are yet not totally sure since $\Delta^4$ is \emph{a priori} only a disc with corners and not necessarily a polytope in some strict sense.

\begin{figure}
 \begin{center}
  \includegraphics[width = 11 cm]{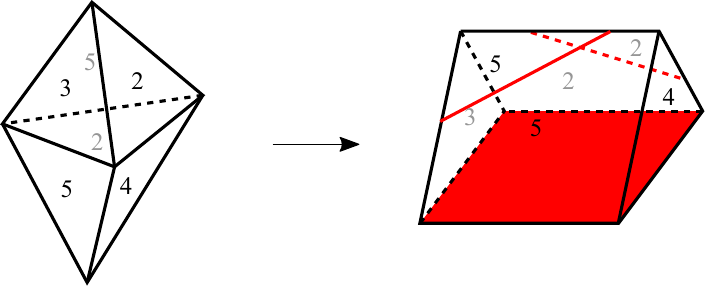}
 \end{center}
 \caption{The effect of the Dehn filling on the facet of $P^4$ dual to the vertex $(0,0,0,1,1)$. The three ideal vertices are replaced with two red edges and one red square. The colours on the back faces are drawn in gray.}
 \label{filledP4:fig}
\end{figure}

We examine the effect of Dehn filling on the facet $F$ that is dual to the vertex $(0,0,0,1,1)$ in the rectified simplex: this is the top vertex in Figure \ref{P4:fig}. By looking at the figure we see that $F$ is coloured (with the colours induced by the adjacent facets) as in Figure \ref{filledP4:fig}-(left). Recall that the Dehn filling of $P^4$ replaces each ideal vertex with an ideal square connecting the two opposite facets with equal colours. The effect of the Dehn filling on $F$ is shown in Figure \ref{filledP4:fig}: the two back ideal vertices of $P^3$ are adjacent to two opposite faces of $F$ with the same colour and are hence replaced with a red edge that separates these two faces; the front ideal vertex is adjacent to four faces with distinct colours and is hence replaced by a square. In Figure \ref{filledP4:fig}-(right) we also  smoothed the two red edges, and we get a triangular prism.

\begin{figure}
\centering
\includegraphics[width=12 cm]{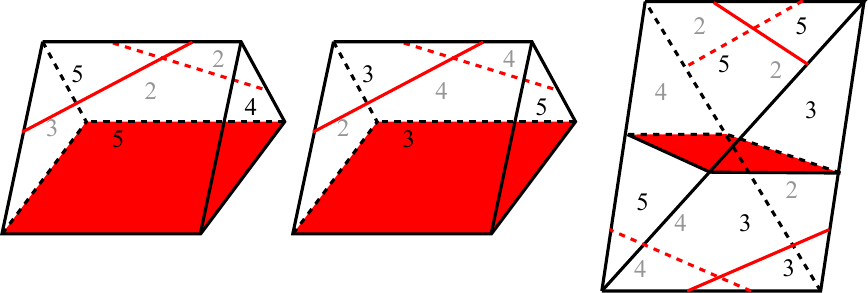}
\caption{The two facets $F,F'$ of $P^4$ coloured with 1 become two triangular prisms after Dehn filling, intersecting in their common red face (left and centre). After smoothing the two prisms glue to a single tetrahedron (right).}
\label{prisms:fig}
\end{figure}

The facet $F$ is coloured with 1. If we examine the other facet $F'$ of $P^4$ coloured with 1 we see an analogous picture where the Dehn filling transforms $F'$ into a prism as in Figure \ref{prisms:fig}-(centre). The two facets $F$ and $F'$ with colour 1 are therefore transformed into the two prisms shown in Figure \ref{prisms:fig}-(left and centre). The two prisms intersect in their common red face. The smoothing glues the two facets as shown in Figure \ref{prisms:fig}-(right) and the final result is a tetrahedron. The four vertices of the tetrahedron are those that were adjacent in $P^4$ to some facet with colour 1.

Everything is equivariant under $2\pi/5$ rotations of the graph shown in Figure \ref{P4:fig}, so every pair of facets of $P^4$ with the same colour $i\in \{1,\ldots,5\}$ will glue to a tetrahedron in the final polytope $\Delta^4$ whose four vertices are those that were adjacent to a $i$-coloured facet in $P^4$. We finally deduce that $\Delta^4$ is a simplex.
\end{proof}

Summing up, by Dehn filling the coloured $P^4$ we get the coloured simplex $\Delta^4$, that we interpret as the spherical right-angled simplex 
$$\Delta^4 = S^4 \cap \{x_1,\ldots, x_5 \geq 0\}$$
whose $i$-th facet $\Delta^4 \cap \{x_i = 0\}$ is coloured with $i$. The colouring of $\Delta^4$ gives rise to $S^4$, decomposed into $2^5 = 32$ copies of $\Delta^4$ in the standard way. Therefore we have obtained $S^4$ as a Dehn filling of $M^4$, as promised.

\subsubsection{The link of tori} \label{link:subsusection}
We now determine the link $L'\subset S^4$ formed by the core tori of the Dehn filling, and show that it is isotopic to the \Ivan link $L\subset S^4$. These core tori are obtained as the union of all the red squares that we inserted in place of the ideal vertices in our Dehn filling process. The $i$-th facet of $\Delta^4$ contains a middle red square as in Figure \ref{prisms:fig}. The red square also inherits a colouring of its edges (from the faces that contain them), with 4 different colours. All the copies of this red squares form a torus tessellated into $2^4 = 16$ red squares, entirely contained in the $i$-th coordinate 3-sphere $x_i=0$. 

We can visualize the intersection $L'\cap \{x_1=0\}$ by examining the first facet of $\Delta^4$, which is a right-angled spherical tetrahedron: this intersects $L'$ as in Figure \ref{prisms:fig}-(right), depicted more geometrically in Figure \ref{spherical:fig}-(left). By mirroring it along three of its faces we get a disc as in Figure \ref{spherical:fig}-(right), and by mirroring it along its boundary sphere we get the picture shown in Figure \ref{torus_all:fig}, as required.

\begin{figure}
\centering
\includegraphics[width=11 cm]{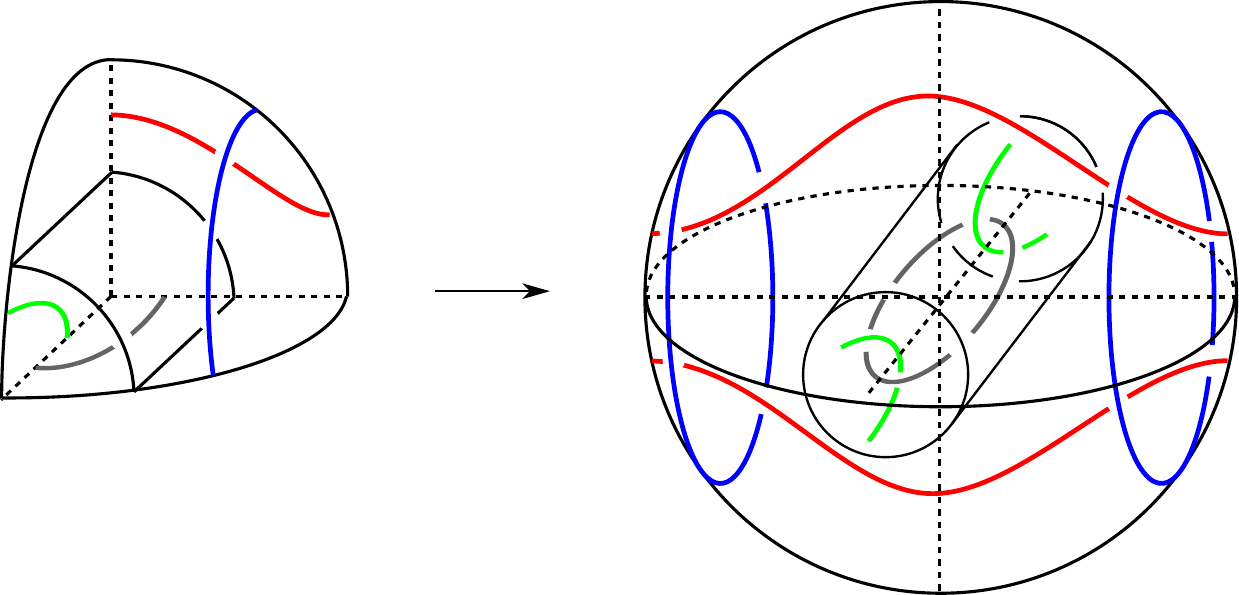}
\caption{The first face of $\Delta^4$ is a right-angled spherical tetrahedron (left). The four red edges and the middle red square from Figure \ref{prisms:fig}-(right) are drawn here as curved objects with different colours.
We mirror the spherical tetrahedron along three of its faces and get a 3-disc (right). The whole 3-sphere $x_1=0$ is obtained by mirroring again the 3-disc along its boundary.}
\label{spherical:fig}
\end{figure}

The right-angled spherical simplex $\Delta^4$ has vertices $e_1,\ldots, e_5$. We remove the vertex $e_1$ and use the coordinate $x_1$ as time: the intersection of $\Delta^4$ with the slice $x_1=t\in [0,1)$ is a tetrahedron, which varies only by rescaling. Therefore we can see the intersection of $L$ with $\Delta^4$ as a film inside a tetrahedron. 

By analysing the intersection of $L$ with the other facets of $\Delta^4$ we check that the film is as in Figure \ref{film:fig}. For instance the facet $x_5=0$ is shown in Figure \ref{film2:fig}: the intersection of the red quadrilateral with the plane $x_1=t$ is a film in the triangle with vertices $e_2,e_3,e_4$ as in Figure \ref{film:fig}.
By doubling the film of Figure \ref{film:fig} along the faces we get the film of Figure \ref{tori:fig}. Therefore $L'$ is isotopic to the \Ivan link $L$, as promised.

\begin{figure}
\centering
\vspace{.2 cm}
\labellist
\pinlabel $e_2$ at 0 -10
\pinlabel $e_5$ at 115 -10
\pinlabel $e_4$ at 140 140
\pinlabel $e_3$ at 10 140
\pinlabel $t=0$ at 50 -10
\pinlabel $t=0.4$ at 200 -10
\pinlabel $t=0.6$ at 345 -10
\pinlabel $t=0.8$ at 490 -10
\endlabellist
\includegraphics[width=12 cm]{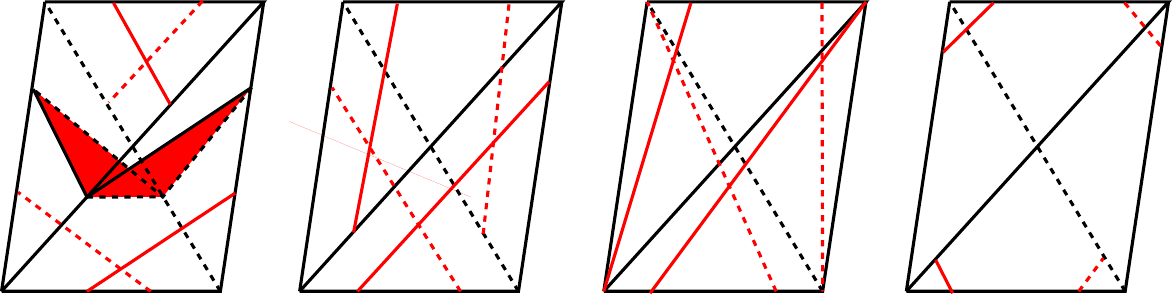}
\vspace{.2 cm}
\caption{We can visualize the intersection of the link $L$ with the 4-simplex $\Delta^4$ as a film on a tetrahedron.}
\label{film:fig}
\end{figure}

\begin{figure}
\centering
\labellist
\pinlabel $e_2$ at 40 -10
\pinlabel $e_4$ at 190 140
\pinlabel $e_3$ at 50 140
\pinlabel $e_1$ at -10 220
\endlabellist
\includegraphics[width=3 cm]{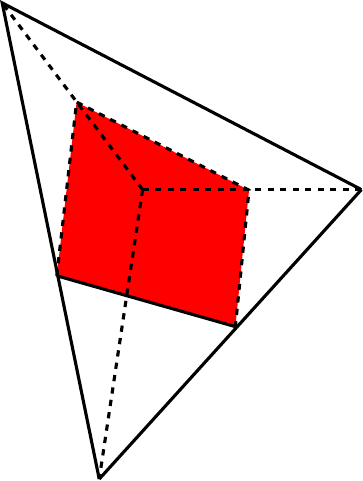}
\vspace{.2 cm}
\caption{The red quadrilateral in the face $x_5=0$ of $\Delta$.}
\label{film2:fig}
\end{figure}



\section{Geometric properties of $L$} \label{geometric:section}
We prove here the points (3), (4) and (9) of Theorem \ref{main:teo}.

\subsection{Symmetries} \label{symmetries:subsection}
We proved in Section \ref{link:subsusection} that we can isotope the \Ivan link $L$ to the link formed by the red core tori of the Dehn filling. After this isotopy we have
$$L = T_1 \sqcup T_2 \sqcup T_3 \sqcup T_4 \sqcup T_5$$
where $T_i$ is the torus that contains the red square in the $i$-th facet of $\Delta^4$. The torus $T_i$ is contained in the 3-sphere $x_i=0$.

\begin{prop}
The isometries
\begin{equation} \label{isom:eqn}
(x_1, x_2, x_3, x_4, x_5) \longmapsto (\pm x_{\sigma(1)}, \pm x_{\sigma(2)}, \pm x_{\sigma(3)}, \pm x_{\sigma(4)}, \pm x_{\sigma(5)})
\end{equation}
where $\sigma\in S_5$ is any permutation contained in the order 20 group $H$ generated by the cycles $(12345)$ and $(2354)$, and signs $\pm$ are arbitrary,
preserve $L$ and restrict to an isometry of the hyperbolic manifold $S^4 \setminus L$.
\end{prop}
\begin{proof}
We first note that any map $(x_1,\ldots,x_5) \mapsto (\pm x_1, \ldots, \pm x_5)$ preserves $L$ because $L$ is obtained by mirroring its intersection with the right-angled simplex $\Delta^4 = S^4 \cap \{x_i \geq 0 \}$ and so it is symmetric by reflection along the coordinate hyperplanes by construction. 

Then we can easily check that $P^4$ has a couple of isometries that preserve the colouring (as a partition of its faces) and act on the colours as the cycles $\sigma_1 = (12345)$ and $\sigma_2= (2354)$. The first is just a rotation of Figure \ref{P4_colour:fig}, the second is less evident from the picture (we have found it using the code available in \cite{code}) and can be checked by looking at the adjacencies between vertices: it sends every vertex in the exterior pentagon to one in the interior one (and viceversa). The two colouring-preserving isometries of $P^4$ extend to isometries of $M$, and also to its Dehn filling: here $\sigma_1, \sigma_2$ act on $\Delta^4$ by permuting its vertices, and preserving its intersection with $L$, and extend to isometries of $S^4$ that preserve $L$. 
\end{proof}

The isometry \eqref{isom:eqn} sends $T_i$ to $T_{\sigma(i)}$ for every $i=1,\ldots, 5$. Let $G$ the group of $20\times 32 = 640$ isometries \eqref{isom:eqn}. Since $H$ acts $k$-transitively on $\{1,2,3,4,5\}$, the group $G$ acts $k$-transitively on the tori $T_i$. We have proved Theorem \ref{main:teo}-(3).

\subsection{The double branched covering} \label{branched:subsection}
We take inspiration from the Borromean rings discussed in Section \ref{double:B:subsubsection} to study the double branched covering of $S^4$ along $L$. The crucial fact is the following.

\begin{figure}
\centering
\includegraphics[width=9 cm]{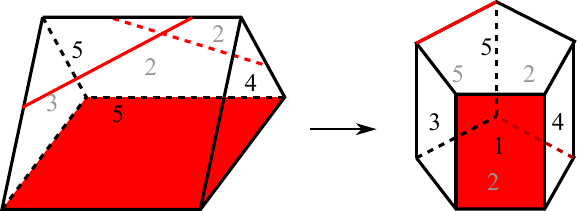}
\caption{These two polyhedra are combinatorially equivalent.}
\label{pentagons:fig}
\end{figure}

\begin{prop} \label{penta:prop}
The polytope $\bar P^4$ obtained in Section \ref{Dehn:4:subsubsection} by Dehn filling $P^4$ is the product of two pentagons. 
\end{prop}
\begin{proof}
We first note that $\bar P^4$ has 10 facets (those of $P^4$) and 25 vertices (the 5 real vertices of $P^4$, plus $5\times 4$ new vertices arising from the transformation of 5 ideal vertices into 5 squares), like the product ot two pentagons. 

We know from Figure \ref{filledP4:fig} that the Dehn filling transforms the facet of $P^4$ dual to the vertex $(0,0,0,1,1)$ into the polyhedron shown in Figure \ref{pentagons:fig}-(left). This polyhedron (considered without smoothing its red edges) is in fact combinatorially equivalent to a pentagonal prism, see Figure \ref{pentagons:fig}-(right). This holds for all the facets of $P^4$. Therefore $\bar P^4$ has 10 facets, and each facet is a prism with pentagonal basis.

The adjacency graph of the facets of $\bar P^4$ is obtained from that of $P^4$ shown in Figure \ref{P4_colour:fig} by adding an edge joining every pair of facets with the same colour. A vertex that lies in the exterior pentagon of the figure is thus connected to all the vertices in the interior pentagon, and viceversa. This is the adjacency graph of a product of two pentagons. 
\end{proof}

We can represent $\bar P^4$ geometrically as a product of two \emph{right-angled hyperbolic pentagons}, thus as a \emph{right-angled polytope} in $\matH^2 \times \matH^2$. The same reasoning of Section \ref{double:B:subsubsection} then shows that we may give to $S^4$ the structure of a $\matH^2 \times \matH^2$ \emph{$\pi$-orbifold} with singular locus $L$ having cone angle $\pi$. The branched double cover $W$ over $S^4$ branched over $L$ inherits the $\matH^2 \times \matH^2$ geometry, without the singular locus: it is an honest manifold with geometry $\matH^2 \times \matH^2$. This proves Theorem \ref{main:teo}-(4). 

\subsection{Five quadrics in $S^4$} 
We now note a curious fact: the \Ivan link $L\subset S^4$ may be obtained by perturbing a very simple set of \emph{five quadrics} that intersect inessentially in finitely many points.
Consider for each $i$ the quadric 
$$Q_i = \big\{x_i = 0,\ x_{i+1}^2 -x_{i+2}^2 - x_{i+3}^2 + x_{i+4}^2 = 0\big\}.$$

Indices are considered modulo 5. The quadric is a torus in the 3-sphere $x_i=0$. As $i=1,\ldots, 5$ varies we get 5 tori, which intersect pairwise in points: for instance the tori $i=1$ and $i=2$ intersect in the four points
$$\{(0,0,\pm \sqrt 2/2, 0,\pm \sqrt 2 / 2)\}.$$

As we soon see, these intersections are inessential, \emph{i.e.}~they can be destroyed by a small isotopy. Note that the quadrics are preserved by the group $G$ defined in Section \ref{symmetries:subsection}.
Figure \ref{quadric:fig} shows the intersection of the \Ivan link and of the quadrics with the first face of $\Delta^4$. The quadrics intersect in pairs at the midpoints of the edges, while the tori in $L$ do not intersect. The left picture is clearly obtained by perturbing the right one. The same picture holds for all the faces of the copies of $\Delta^4$ that triangulate $S^4$ and we can perturb them simultaneously. The \Ivan link $L$ is obtained by perturbing the five quadrics in order to destroy the intersections.

\begin{figure}
\centering
\vspace{.3 cm}
\labellist
\pinlabel $e_2$ at -10 -10
\pinlabel $e_5$ at 120 -10
\pinlabel $e_3$ at 10 150
\pinlabel $e_4$ at 140 150
\pinlabel $e_2$ at 175 -10
\pinlabel $e_5$ at 305 -10
\pinlabel $e_3$ at 195 150
\pinlabel $e_4$ at 325 150
\endlabellist
\includegraphics[width=8 cm]{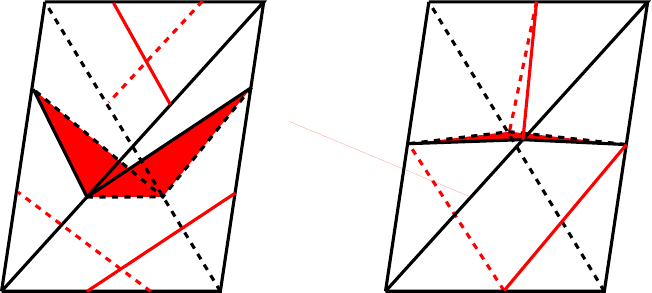}
\vspace{.2 cm}
\caption{The first face of $\Delta^4$ intersects the \Ivan link and the quadrics into red discs and arcs as shown here on the left and right, respectively (pictures are accurate only up to isotopy). The left picture is obtained by perturbing the right one in order to destroy the intersections between discs and arcs.}
\label{quadric:fig}
\end{figure}

\subsection{Perfect position} \label{perfect:subsection}
We prove that $L\subset S^4$ can be isotoped to be in perfect position with respect to a height function. As a height function on $S^4$, we consider the diagonal map
$$f(x_1,x_2,x_3,x_4,x_5) = x_1+x_2+x_3+x_4+x_5.$$

We have just shown that each torus $T_i$ in $L$ is obtained by perturbing the quadric $Q_i$.
We can check easily that $f$ restricts to a perfect Morse function on $Q_i$. Its restriction on $Q_1$ has one minimum at $(0, -1/2, -1/2, -1/2,-1/2)$, two saddle points at
$(0, \pm 1/2, \pm 1/2, \mp 1/2, \mp 1/2)$, and one maximum at $(0,1/2,1/2,1/2,1/2)$, and the singular points for $Q_i$ are analogous. Since the perturbation to pass from $Q_i$ to $T_i$ is performed away from these points, we can adjust it so that the restriction of $f$ to $T_i$ is also perfect Morse for all $i$. This proves Theorem \ref{main:teo}-(9).

\section{The fundamental group} \label{fundamental:section}

We prove here the points (5), (6), and (8) of Theorem \ref{main:teo}. We will then use the information collected on the Alexander ideal to prove Theorem \ref{Betti:teo}.

\subsection{The fundamental group} \label{pi:subsection}
We determine a presentation for the fundamental group $\pi_1(M)$. To this purpose we use Bestvina -- Brady theory \cite{BB} to construct a convenient function $f\colon C \to \matR$ on a cube complex $C$ homotopy equivalent to $M$. The map $f$ will look like a Morse function and we will read a presentation from it.

The manifold $M$ decomposes into copies $P^4_v$ of $P^4$ parametrised by $v \in \matZ_2^5$. The dual cubulation $C$ of the tessellation of $M$ into $P^4_v$ has $2^5$ vertices, dual to the polytopes $P^4_v$ and hence identified with the elements in $\matZ_2^5$, and then edges, squares, cubes, and hypercubes duals to the 3-faces, 2-faces, edges, and real vertices of the tessellation of $M$ into $P^4_v$. We embed $C$ in $M$ as a \emph{spine} as explained in \cite{BM, IMM}. The manifold $M$ deformation retracts onto $C$, the complement $M\setminus C$ consists of (truncated) cusps.

\begin{figure}
 \begin{center}
  \includegraphics[width = 12 cm]{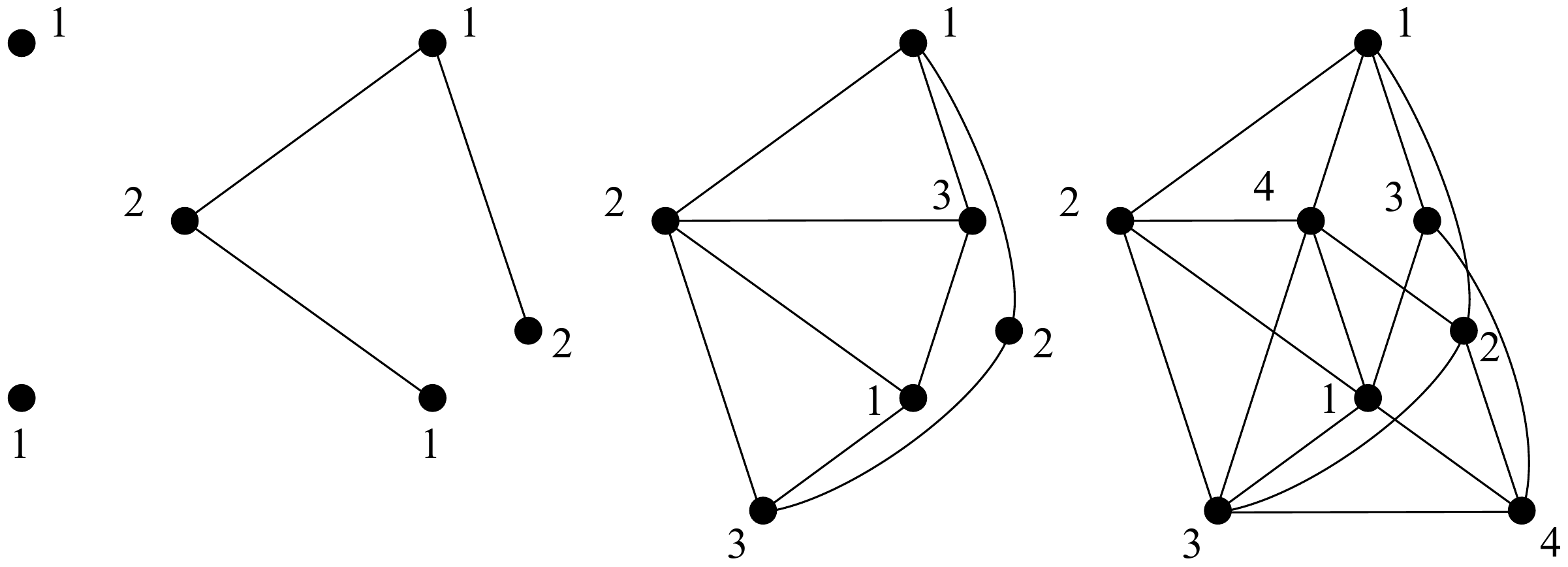}
 \end{center}
 \caption{The descending links of the vertices $v$ with $f(v)=v_1+v_2+v_3+v_4+v_5=1,2,3,4$ respectively, is the flag complex spanned by the graph shown here.}
 \label{descending_links:fig}
\end{figure}


We define a piecewise-linear function $f\colon C\to \matR$ by setting $f(v)=v_1+\cdots +v_5$ and extending $f$ linearly to all the cubes of $C$. This is a Bestvina -- Brady Morse function, and hence we may apply the theory of \cite{BB} to investigate its singular points similarly as in \cite{BM, IMM}.

The function $f$ has a unique minimum at $0$ where $f(0)=0$ and a unique maximum at $v=(1,1,1,1,1)$ where $f(v) = 5$. The descending link at the vertex $v\in \matZ_2^5$ is the subcomplex of the rectified simplex generated by the vertices whose colour $i$ is such that $v_i=1$. Up to isomorphism, the descending link depends only on the number $f(v) = v_1+\cdots +v_5 \in\{0,\ldots, 5\}$ of such colours. The cases with $f(v)=1,2,3,4$ are shown in Figure \ref{descending_links:fig}. The descending link of a vertex $v$ with $f(v)=1,2,3,4,5$ is homotopic respectively to two points, a point, a circle, a point, and a bouquet of four spheres. 

We are interested in the sublevel sets $f_t = f^{-1}[0,t]$. The sublevel set $f_t$ is contractible for $t<1$, and every time we cross a critical point $t=1,2,3,4,5$ it changes by coning the descending links of the vertices $v$ with $f(v)=t$. Note that there are 5, 10, 10, 5, 1 vertices $v$ with $f(v)=1,2,3,4,5$. The homotopy type of $f_t$ changes by adding five 1-cells at $t=1$, ten 2-cells at $t=3$, and four 3-cells at $t=5$. In particular $\pi_1(C)$ has a presentation with five generators and ten relators. Note that all these numbers are the minimum required by homology since $b_1(C) = 5, b_2(C)=10, b_3(C)=4$. (The spine $C$ is homotopically equivalent to $M$.)

\begin{figure}
 \begin{center}
  \includegraphics[width = 5.5 cm]{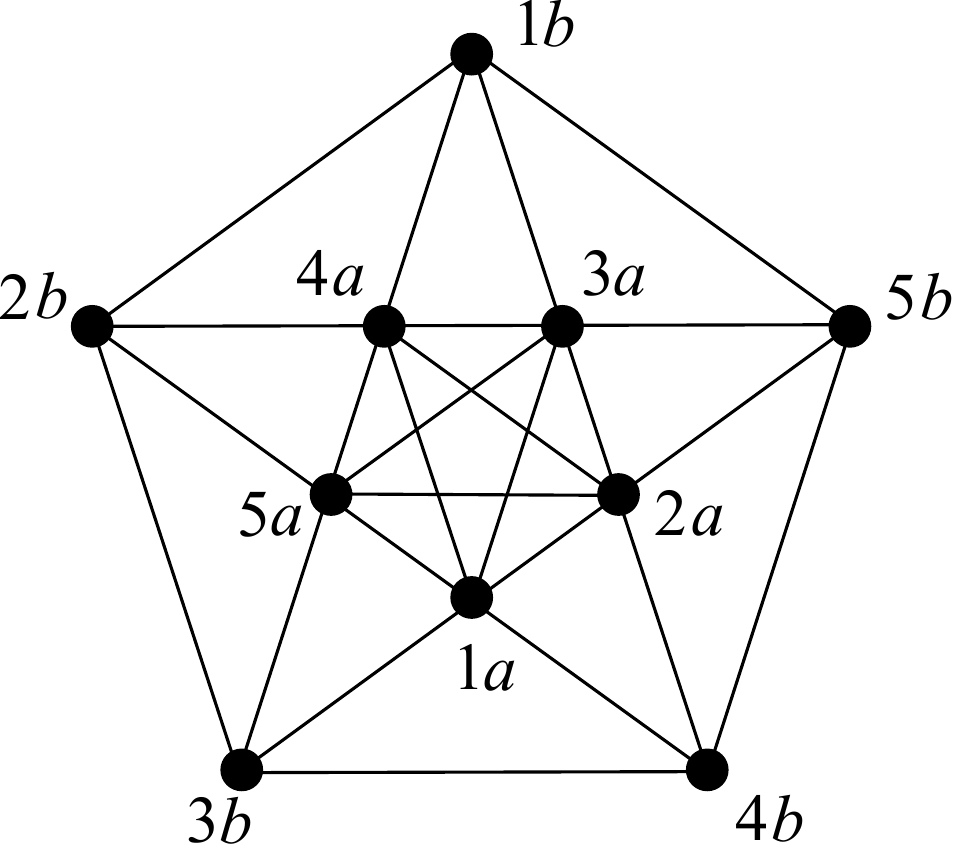}
 \end{center}
 \caption{A convenient labeling of the facets of $P^4$.}
 \label{P4_labels:fig}
\end{figure}

We now analyse carefully the cubulation $C$ and the function $f$ to identify generators and relators for $\pi_1(C)$. It is convenient to label the 10 facets of $P^4$ as in Figure \ref{P4_labels:fig} with the symbols $1a, 1b, \ldots, 5a, 5b$. Every edge of $C$ is dual to some facet of $P^4$ and is hence also assigned that label. All the edges of $C$ are thus labeled; note that opposite edges along a square of $C$ receive the same label, while the 10 edges adjacent to the same vertex $v$ of $C$ are all assigned different labels (the link of $v$ is as in Figure \ref{P4_labels:fig}).

Every string of labels like $1a\,3a\,1a\,3a$ determines a path in the 1-skeleton of $C$, that starts at the base vertex $0$ and ends at some vertex $v$. The path is obtained by starting at $0$, walking through the edge $1a$ and getting to some vertex $v_1$, then walking through the unique edge $3a$ incident to $v_1$ to reach some other vertex $v_2$, and so on. If the path is closed, it determines an element of $\pi_1(M)$. A closed path consisting of two repeated labels like $1a\,1a$ or $3b\,3b$ is homotopically trivial. The paths
$$1a\,1b, \quad 2a\,2b, \quad 3a\,3b,\quad 4a\,4b, \quad 5a\,5b$$
are closed and determine the 5 generators of $\pi_1(M)$ produced by crossing the time $t=1$, see Figure \ref{descending_links:fig}-(left). We now identify the ten relators produced at $t=3$.

When we pick two labels of two adjacent facets, like $1a$ and $3a$, the closed path $1a\,3a\,1a\,3a\,$ bounds a square in $C$ and is hence homotopically trivial. If we pick two non-adjacent facets, like $1b$ and $3b$, the closed path $1b\,3b\,1b\,3b$ does not bound a square and is more interesting. It turns out that the element $1b\,3b\,1b\,3b$ in $\pi_1(C)$ is the commutator of the inverted generators $1b\,1a$ and $3b\,3a$. Here is the proof:
$$1b\,1a\,3b\,3a\,1a\,1b\,3a\,3b \sim 
1b\,3b\,1a\, 1a\, 3a\,1b\,3a\,3b \sim 
1b\,3b\,1b\,3a\,3a\,3b \sim 
1b\,3b\,1b\,3b.
$$
We have used that $1a\, 3b \sim 3b\, 1a$, $3a\, 1a \sim 1a\, 3a$, and $3a\, 1b \sim 1b\, 3a$ since in all three cases the two facets are adjacent. Another useful observation concerning $1b\,3b\,1b\,3b$ is that it commutes with the generator $2a\, 2b$, since each of $1b, 3b$ commutes with each of $2a, 2b$. By combining these two facts we have found the relation
\begin{equation} \label{relator:eqn}
[2a\,2b, [1b\,1a, 3b\,3a]].
\end{equation}
This is in fact the relation that appears when we cross $t=3$ and we cone the
circle $1b, 2a, 3b, 2b$ in Figure \ref{descending_links:fig}. To show this, note that the circle in the descending link is homotopic to the octagon
$$2a\, 1b\, 3b\, 2a\, 2b\, 3b\, 1b\, 2b$$
based at the initial vertex $v = (0,1,1,0)$ and connecting the vertices 
$$(0,1,1,0), (0,0,1,0), (1,0,1,0), (1,0,0,0), (1,1,0,0), (1,0,0,0), (1,0,1,0), (0,0,1,0).$$
We connect $0$ and $v$ via the arc $3b \, 2a$ and we use it to homotope the octagon to
$$3b\, 1b\, 3b\, 2a\, 2b\, 3b\, 1b\, 2b\, 2a\, 3b$$
that is in turn homotopic to $[3b\,1b\,3b\,1b,2a\,2b]$.

The 10 relators that appear at $t=3$ are like \eqref{relator:eqn} and by renaming generators as
$$a = 1a\,1b, \quad b= 2a\,2b, \quad c= 3a\,3b,\quad d= 4a\,4b, \quad e= 5a\,5b$$
we find the 10 relators of Theorem \ref{main:teo}-(5), whose proof is now complete.

Recall that Figure \ref{P4_facet:fig} represents an octahedral facet of the rectified simplex, and together with its colouring a 3-torus cusp section of $M$. By looking at the figure we notice that the three curves 
$$1b\, 3b\, 1b\, 3b, \quad 4a\, 5a\, 4a\, 5a, \quad 2a\, 2b$$
represent the three generators of the fundamental group of this 3-torus section. Therefore, as stated in the introduction, the elements of the fundamental group
$$[a^{-1}, c^{-1}], \quad [d,e], \quad b$$
generate the peripheral group $\matZ^3$ of one cusp. Generators for the other cusps are obtained from these by permuting $a,b,c,d,e$ cyclically.

\subsection{The Alexander ideal} \label{Fox:subsection}
Given the nice representation of $\pi_1(M)$ with 5 generators $a,b,c,d,e$ and 10 relators, it is routine to determine the Alexander ideal using Fox calculus \cite{F}. This can be done by Sage \cite{Sage} via the following script 

{\small
\begin{verbatim}
  G.<a,b,c,d,e> = FreeGroup() 
  R = [a*c*d/c/d/a*d*c/d/c, b*d*e/d/e/b*e*d/e/d, c*e*a/e/a/c*a*e/a/e, 
     d*a*b/a/b/d*b*a/b/a, e*b*c/b/c/e*c*b/c/b] 
  S = [a/b/e*b*e/a/e/b*e*b, b/c/a*c*a/b/a/c*a*c, c/d/b*d*b/c/b/d*b*d, 
     d/e/c*e*c/d/c/e*c*e, e/a/d*a*d/e/d/a*d*a] 
  T = R + S
  H = G.quotient(T)
  R.<t1,t2,t3,t4,t5> = LaurentPolynomialRing(ZZ)
  A = H.alexander_matrix([t1,t2,t3,t4,t5])
  minors = A.minors(4)
\end{verbatim}
}

We get the $10\times 5$ matrix
\setlength\arraycolsep{-7pt}
$${\tiny 
\begin{bmatrix}
0& 0& -  (t_{4} - 1)  (t_{1} - 1) & (t_{3} - 1)  (t_{1} - 1)&
  0&\\
 0&
  0&
  0&
  -  (t_{5} - 1)  (t_{2} - 1)&
  (t_{4} - 1)  (t_{2} - 1)\\
 (t_{5} - 1)  (t_{3} - 1)&
  0&
  0&
  0&
  -  (t_{3} - 1)  (t_{1} - 1)\\
 -  (t_{4} - 1)  (t_{2} - 1)&
  (t_{4} - 1)  (t_{1} - 1)&
  0&
  0&
  0\\
 0&
  -  (t_{5} - 1)  (t_{3} - 1)&
  (t_{5} - 1)  (t_{2} - 1)&
  0&
  0\\
 0&
  -t_{2}^{-1} t_{5}^{-1} (t_{5} - 1)  (t_{1} - 1)&
  0&
  0&
  t_{2}^{-1} t_{5}^{-1}  (t_{2} - 1)  (t_{1} - 1) \quad \\
 t_{1}^{-1} t_{3}^{-1}  (t_{3} - 1)  (t_{2} - 1)&
  0&
  -t_{1}^{-1} t_{3}^{-1}  (t_{2} - 1)  (t_{1} - 1)&
  0&
  0\\
 0&
  t_{2}^{-1} t_{4}^{-1}  (t_{4} - 1)  (t_{3} - 1)&
  0&
  -t_{2}^{-1} t_{4}^{-1}  (t_{3} - 1)  (t_{2} - 1)&
  0\\
 0&
  0&
  t_{3}^{-1} t_{5}^{-1}  (t_{5} - 1)  (t_{4} - 1)&
  0&
  \!\!\! -t_{3}^{-1} t_{5}^{-1}  (t_{4} - 1)  (t_{3} - 1) \quad \quad \\
 -t_{1}^{-1} t_{4}^{-1}  (t_{5} - 1)  (t_{4} - 1)&
  0&
  0&
  t_{1}^{-1} t_{4}^{-1}  (t_{5} - 1)  (t_{1} - 1)&
  0
 \end{bmatrix}.
 }
 $$

Its 1050 minors are either null or (up to invertible factors $\pm t_i^{\pm 1}$) of type
$$(t_1-1)^{a_1}(t_2-1)^{a_2}(t_3-1)^{a_3}(t_4-1)^{a_4}(t_5-1)^{a_5}$$
with $0\leq a_i \leq 4, a_1+\cdots + a_5 = 8$, where at most one $a_i$ vanishes, and the conditions $(a_{i-1},a_i,a_{i+1}), (a_{i-2},a_i,a_{i+2}) \neq (1,0,1) $ are satisfied for all $i$. The indices should be interpreted cyclically modulo 5. These monomials are by definition the generators of the Alexander ideal $I$.

The Alexander polynomial $\Delta$ is the greatest common divisor of these monomials, which is clearly 1 because the exponent $a_i$ is allowed to vanish. Therefore $\Delta = 1$. This proves Theorem \ref{main:teo}-(6).

\subsection{Infinite cyclic coverings} \label{infinite:subsection}
We prove here Theorem \ref{Betti:teo}. We start with a general fact concerning hyperbolic 4-manifolds.

\begin{prop}
Every infinite cyclic covering $\tilde M$ of a finite-volume hyperbolic 4-manifold $M$ has $b_2(\tilde M) = \infty$.
\end{prop}
\begin{proof}
We have $M = \tilde M /\tau$ with deck automorphism $\tau$. The manifold $M_d = \tilde M /\tau^d$ covers $M$ with degree $d$. Since $\chi(M) >0$ we get $\chi(M_d) \to \infty$, hence $b_2(M_d) \to \infty$. This easily implies that $b_2(\tilde M) = \infty$, as follows.

Pick a smooth function $f\colon \tilde M \to \matR$ such that $f(\tau(x)) = f(x)+1$ (lift a map $M \to S^1$ that represents the homomorphism $\pi_1(M) \to \matZ$ induced by the covering). Set $X_d = f^{-1}(0,d)$. The covering $\pi\colon \tilde M \to M_d$ sends $X_d$ and $f^{-1}(-\varepsilon, \varepsilon)$ homeomorphically to two open subsets of $M_d$ that cover $M_d$. Using the Mayer -- Vietoris sequence of this covering together with $b_2(M_d) \to \infty$ we find that 
$$\dim \pi_*(j_*(H_2(X_d))) = \dim i_*(\pi_*(H_2(X_d))) = \dim i_*(H_2(\pi(X_d))) \to \infty$$
where $j\colon X_d \hookrightarrow \tilde M$ and $i\colon \pi(X_d) \hookrightarrow M_d$ are the inclusion maps. 
We deduce that $\dim  j_*(H_2(X_d)) \to \infty$  and hence $b_2(\tilde M) = \infty$.
\end{proof}

We now turn back to our manifold $M= S^3\setminus L$ and fix a primitive class $\phi =(x_1,x_2,x_3,x_4,x_5)\in H^1(M,\matZ)=\matZ^5$. 

We first calculate $b_3(\tilde M)$. If $x_i \neq 0$ for all $i$, then $\phi$ is represented by a Morse function with only critical points of index 2 by \cite{BM}. This lifts to a Morse function $f\colon \tilde M \to \matR$ with infinitely many critical points of index 2, and whose regular fiber is a the interior of a compact 3-manifold with boundary. This implies that $\tilde M$ has a handle decomposition with only 0-, 1-, and (infinitely many) 2-handles. Therefore $b_3(\tilde M) = 0$. Conversely, if $x_i=0$ the class $\phi$ vanishes on the $i$-th peripheral $\matZ^3$ subgroup, hence $\tilde M$ has infinitely many cusps of rank three, its sections are independent elements in $H_3(\tilde M)$, hence $b_3(\tilde M) = \infty$.

We are left with the most interesting part, that is the calculation of $b_1(\tilde M)$.
By sending $t_i$ to $t^{x_i}$ for all $i$ we get a ring homomorphism 
$$\matZ[t_1^{\pm 1}, \ldots, t_5^{\pm 1}] \to \matZ[t^{\pm 1}]$$ 
that sends
the Alexander ideal $I$ to some ideal $I_\phi \subset \matZ[t^{\pm 1}]$. Let the Laurent polynomial $\Delta_\phi\in \matZ[t^{\pm 1}]$ be the greatest common divisor of $I_\phi$. It is well-defined up to multiplication with $\pm t^k$, and we let its \emph{degree} be the difference between its highest and lowest exponent, or $\infty$ if $\Delta_\phi=0$. 

Let $\tilde M$ be the infinite cyclic covering associated to $\ker \phi$. The following equality is due to Milnor \cite{Mil, McM}:
$$b_1(\tilde M) = \deg \Delta_\phi.$$
The generators of the Alexander ideal $I$ are mapped to generators of $I_\phi$. These are
\begin{equation} \label{mono:eqn}
(t^{x_1}-1)^{a_1}(t^{x_2}-1)^{a_2}(t^{x_3}-1)^{a_3}(t^{x_4}-1)^{a_4}(t^{x_5}-1)^{a_5}
\end{equation}
with $0\leq a_i \leq 4, a_1+\cdots + a_5 = 8$, such that at most one $a_i$ vanishes, and the conditions $(a_{i-1},a_i,a_{i+1}), (a_{i-2},a_i,a_{i+2}) \neq (1,0,1) $ are satisfied for all $i$. 

Suppose that $x_i\neq 0$ for all $i$. The given generators are therefore all nontrivial.
Each $t^{x_i}-1$ decomposes into (irreducible) cyclotomic polynomials $\Phi_d(t)$ as
$$t^{x_i}-1 = \prod_{d|x_i} \Phi_d(t).$$
This is true also for negative $x_i$ up to multiplying with
$-t^{-x_i}$. The product is on all positive divisors $d>0$ of $x_i$.
Since $a_1+\cdots +a_5=8$ and $\Phi_1(t) = t-1$, each generator (\ref{mono:eqn}) decomposes as 
\begin{equation} \label{prod:eqn}
\prod_{i=1,\ldots, 5}\, \prod_{d|x_i} \Phi_{d}(t)^{a_i} = 
(t-1)^8 \prod_{i=1,\ldots, 5} \,\prod_{d|x_i, d>1} \Phi_{d}(t)^{a_i}.
\end{equation}
Therefore
$$\Delta(t) = (t-1)^8 q(t)$$
for some polynomial $q(t)\neq 0$. This shows that $b_1(\tilde M) = 8 + d$ where $d = \deg q(t)\geq 0$. The polynomial $q(t)$ depends on $x$. If $x_1,\ldots, x_5$ are pairwise coprime we easily see that $q(t)=1$. In this case for each irreducible factor $\Phi_d(t)$ in (\ref{prod:eqn}) there is a unique $x_i$ such that $d|x_i$, hence the factor is present if and only if $a_i>0$, and we can always find a generator with $a_i=0$ that does not contain it: therefore $q(t)=1$. On the other hand, if $d>1$ divides both $x_i$ and $x_j$, the factor $\Phi_d(t)$ will always be present in (\ref{prod:eqn}) because $a_i$ and $a_j$ are not allowed to both vanish. Summing up, we have $d=\deg q(t)=0$ if and only if the numbers $x_1,\ldots, x_5$ are pairwise coprime.
On the other hand, if we pick
$$(x_1,x_2,x_3,x_4,x_5) = (p,p,1,1,1)$$
for some prime number $p$ we discover that each generator (\ref{mono:eqn}) can written as
$$(t-1)^8 \Phi_p(t)^{a_1+a_2}$$
and since $a_1$ and $a_2$ cannot both vanish we get $q(t) = \Phi_p(t)$ and 
$$b_1(\tilde M) = 8+\deg \Phi_p(t) = 7+p.$$

It remains to consider the case $x_i=0$ for some $i$. If there is another $x_j$ that vanishes, then all generators \eqref{mono:eqn} vanish, we get $\Delta_\phi = 0$ and $b_1(\tilde M) = \infty$. If $x_j\neq 0$ for all $j\neq i$, as above we get $b_1(\tilde M) = 8 + d$ for some finite computable $d \geq 0$.

This concludes the proof of Theorem \ref{Betti:teo}.

\subsection{Dehn surgery} \label{Dehn:subsection}
We prove Theorem \ref{main:teo}-(8).
We know from Theorem \ref{main:teo}-(5) that the peripheral groups of $M$ are generated by 
$$a, \quad [c,d], \quad  [e^{-1}, b^{-1}]$$
and the triples obtained by permuting $a,b,c,d,e$ cyclically. The element $a$ is a meridian, while $[c,d]$ and $[e^{-1}, b^{-1}]$ are longitudes. If we kill the longitudes
$$[c,d], \quad [d,e], \quad [e,a], \quad [a,b], \quad [b,c]$$
we get a closed manifold $X$ whose fundamental group is manifestly $\matZ^5$. This proves Theorem \ref{main:teo}-(8).

\section{Link of tori in other closed 4-manifolds} \label{link:section}

We prove here Theorem \ref{Lagrangian:teo} and show that the complement $M=S^4\setminus L$ of the \Ivan link is indeed the double cover of the non-orientable hyperbolic 4-manifold 1011 from the Ratcliffe -- Tschantz census \cite{RT}, which was already proved to be a link complement in $S^4$ by \Ivan \cite{I, I2}.

\subsection{The real projective space} \label{RT:subsection}
By Theorem \ref{main:teo}-(3) the antipodal map of $S^4$ preserves $L\subset S^4$, hence it projects it to a link $\pi(L) \subset \matRP^4$. The complement $\matRP^4 \setminus \pi(L)$ is a non-orientable cusped hyperbolic 4-manifold doubly covered by $M$.

\begin{prop} \label{RT:prop}
The complement of $\pi(L)\subset \matRP^4$ is 
the non-orientable manifold 1011 from the Ratcliffe -- Tschantz census \cite{RT}.
\end{prop}
\begin{proof}
The manifolds in the Ratcliffe -- Tschantz census are by construction all the \emph{small covers} of $P^4$, that is the manifold covers of the orbifold $P^4$ of smallest possible index $2^4=16$. All the small covers can be constructed by \emph{colouring} the facets of $P^4$ with vectors $v \in (\matZ/2\matZ)^4$, as explained for instance in \cite{KMT, KS}. 

We assign the vectors $e_1,e_2,e_3,e_4,e_1+e_2+e_3+e_4$ to the facets of $P^4$ coloured with 1, 2, 3, 4, 5 in  Figure \ref{P4_colour:fig}. This builds a small covering, and it yields the manifold 1011: to see this, we can either note that (i) it has 320 symmetries inherited from the 640 symmetries of $L$, or that (ii) it has 5 cusps, and all its cusp sections are the non-orientable flat manifold (denoted as G in \cite{RT}) that fibers over the circle with monodromy $\matr {-1}001$. In either case the manifold 1011 is the only one in the tables \cite{RT} with this property.
\end{proof}

We have proved the part of Theorem \ref{Lagrangian:teo} concerning $\matRP^4$. Since $M$ is the orientable double cover of this complement, we deduce the following.

\begin{cor}
The manifold $M$ is the manifold considered by Ivan\v si\' c \cite{I, I2}.
\end{cor}

\subsection{Lagrangian tori in products of surfaces} \label{Lagrangian:subsection}
We prove here the part of Theorem \ref{Lagrangian:teo} that concerns the product of surfaces.

\begin{figure}
\centering
\includegraphics[width=7 cm]{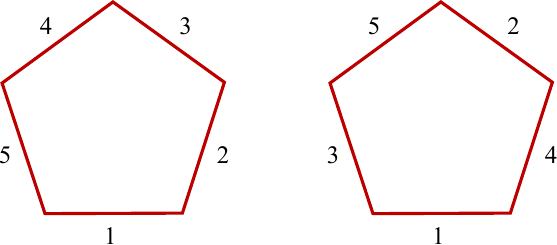}
\caption{The polytope $\bar P^4$ is the product of two right-angled pentagons in $\matH^2$. The 5 red squares are the products of edges with the same labels.}
\label{pentagons_product:fig}
\end{figure}

Coherently with Section \ref{branched:subsection}, we denote with $\bar P^4\subset \matH^2 \times \matH^2$ the product of two right-angled pentagons as in Figure \ref{pentagons_product:fig}. A $k$-face of $\bar P^4$ is the product of a $h$-face of the left pentagon with a $(k-h)$-face of the right one. 
Each edge in Figure \ref{pentagons_product:fig} has a colour in $1,\ldots, 5$. We colour the 5 square faces in $\bar P^4$ that are products of edges with the same colour in \emph{red}. These 5 red squares are pairwise disjoint. 

If we assign a colouring with palette $1,\ldots,k$ to the right-angled polytope $\bar P^4$, we get a manifold $N^4$ with geometry $\matH^2 \times \matH^2$ in the usual way. The union of all the red squares will form some link $L\subset N^4$ of totally geodesic tori.

\begin{prop}
The complement $N^4 \setminus L$ has a hyperbolic structure.
\end{prop}
\begin{proof}
The polytope $\bar P^4 \subset \matH^2 \times\matH^2$ with its red squares was obtained in Section \ref{branched:subsection} as a Dehn filling of $P^4$. Therefore $N^4$ is a Dehn filling of the hyperbolic manifold $N^4 \setminus L$ obtained from $P^4$ equipped with the same colouring of $\bar P^4$. 
\end{proof}

\begin{figure}
\centering
\includegraphics[width=7 cm]{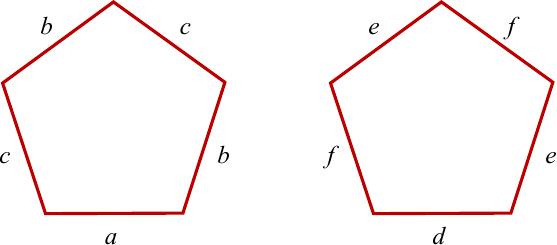}
\caption{A colouring for $\bar P^4$ that is a product of two independent colourings on the pentagons. The resulting manifold is a product of two genus-two surfaces $\Sigma \times \Sigma$.}
\label{pentagons_colour:fig}
\end{figure}

We now build an example. Every facet $F$ of $\bar P^4$ is a product of an edge $e$ of one pentagon with the other pentagon; hence the 10 facets $F$ of $P^4$ are in 1-1 correspondence with the 10 edges of the two pentagons, and we denote a colouring of the facets of $P^4$ by assigning the colours to the corresponding edges. 

Let us colour the facets of $\bar P^4$ with 6 colours as depicted in Figure \ref{pentagons_colour:fig}. We use letters instead of numbers to avoid confusion with Figure \ref{pentagons_product:fig}. Since this is a product of two independent colourings on the pentagons, the resulting manifold is a product $\Sigma \times \Sigma$ of two surfaces, each $\Sigma$ obtained by colouring the right-angled pentagon with 3 colours. The $a,b,c$ colouring produces the genus-two surface $\Sigma$ shown in Figure \ref{genus_two:fig}. The edges coloured with $a,b,c$ form 1, 3, 3 simple closed curves respectively.

\begin{figure}
\centering
\includegraphics[width=8 cm]{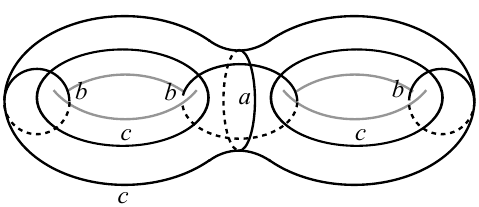}
\caption{A genus two surface obtained from the $a,b,c$ colouring of the pentagon. It is tessellated into 8 identical pentagons with sides labeled as $a,b,c,b,c$. The sides coloured with $a,b,c$ form 1, 3, 3 simple closed curves respectively, and we assign only one label for each closed curve for simplicity.}
\label{genus_two:fig}
\end{figure}

We have drawn the same picture in Figure \ref{genus_two_greek:fig} labeling the simple closed curves as $\alpha, \beta, \gamma, \delta, \eta$ according to the labeling 1, 2, 3, 4, 5 in the left pentagon of Figure \ref{pentagons_product:fig}. Recall that the five red squares in $\bar P^4$ are the products of edges in the left and right pentagons that share the same label. These form the curves 
$$\alpha \times \alpha, \quad \beta \times \gamma_i, \quad \gamma_i \times \epsilon, \quad \delta_i \times \beta, \quad \epsilon \times \delta_i.$$

The proof of Theorem \ref{Lagrangian:teo} is complete.

\subsection{The branched double cover $W$}
We now use the techniques introduced in the previous section to derive some information on the double branched cover $W$ over $S^4$ ramified along the \Ivan link $L$. 

We can calculate the Betti numbers of $W$.

\begin{figure}
\centering
\labellist
\pinlabel $e_1$ at 55 5
\pinlabel $e_1+e_6$ at 215 5
\pinlabel $e_2$ at 110 40
\pinlabel $e_5$ at 0 40
\pinlabel $e_2+e_6$ at 280 40
\pinlabel $e_5+e_6$ at 147 40
\pinlabel $e_3$ at 93 103
\pinlabel $e_4$ at 15 103
\pinlabel $e_3+e_6$ at 260 103
\pinlabel $e_4+e_6$ at 162 103
\endlabellist
\includegraphics[width=7 cm]{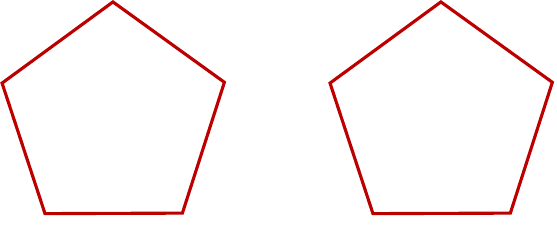}
\caption{This colouring of $\bar P^4$ produces the double branched covering $W$ over $S^4$ ramified over the \Ivan link $L$.}
\label{pentagons_uncoloured:fig}
\end{figure}

\begin{prop}
We have
$$b_0(W) = 1, \quad b_1(W) = 0, \quad b_2(W) = 2, \quad b_3(W) = 0, \quad b_4(W)=1.$$
\end{prop}
\begin{proof}
We can obtain $W$ by assigning to $\bar P^4$ the colouring in $(\matZ/2\matZ)^6$ shown in Figure \ref{pentagons_uncoloured:fig}, see \cite{KMT, KS} for more information on colourings. The manifold is tessellated into the polytopes $\bar P_v$ with $v \in (\matZ/2\matZ)^5$, and the quotient $\pi$-orbifold $(S^4,L)$ with geometry $\matH^2 \times \matH^2$ is constructed by identifying $\bar P_v$ with $\bar P_{v+e_6}$ via the identity map (this is the double branched covering).

By applying the Choi -- Park formula \cite[Theorem 4.6]{CP}, which works for any moment-angle manifold and hence for any manifold obtained by colouring some polytope, we find that the Betti numbers are as stated. One checks that the only complexes $K_\omega$ in that formulathat contribute to the Betti numbers are homeomorphic to $\emptyset, S^1, S^1, S^3$, and they give $b_0=1, b_2=2$, and $b_4=1$.
\end{proof}

\begin{cor} \label{W:not:cor}
The double branched covering $W$ is not a product of surfaces.
\end{cor}

The manifold $W$ is finitely covered by a product of surfaces. For instance, the trivial colouring of $\bar P^4$ that assigns 10 different numbers to the 10 facets produces a product of two surfaces that covers any manifold obtained from $\bar P^4$ by any colouring, including $W$.

\end{document}